\documentclass[10pt]{article}

\usepackage{amsmath,amssymb,amsthm,amsfonts,amstext,amsbsy,amscd}
\usepackage[english]{babel}
\usepackage[utf8]{inputenc}
\usepackage{csquotes}
\usepackage{mathrsfs}
\usepackage{bbm,dsfont}
\usepackage{xcolor}
\usepackage{comment}
\usepackage{mathtools}
\usepackage[arrow, matrix, curve]{xy}
\usepackage{float}
\usepackage{wrapfig}
\usepackage{subcaption}
\usepackage{multirow}
\usepackage{caption}

\usepackage{algorithm}
\usepackage{algpseudocode}

\usepackage{oubraces}

\usepackage[colorlinks,citecolor=blue,urlcolor=blue,linkcolor = blue,hypertexnames=false]{hyperref}
\usepackage{cleveref}

\usepackage[
  a4paper,
  textwidth=6in,
  textheight=8in,
  centering
]{geometry}

\colorlet{colorJP}{orange}

\colorlet{colorMartin}{magenta}

\usepackage{todonotes}

\DeclareMathOperator{\up}{up}
\DeclareMathOperator{\down}{down}

\newcommand\blfootnote[1]{%
  \begingroup
  \renewcommand\thefootnote{}\footnote{#1}%
  \addtocounter{footnote}{-1}%
  \endgroup
}

\theoremstyle{plain}
\newtheorem{thm}{Theorem}
\newtheorem{proposition}{Proposition}
\newtheorem{corollary}{Corollary}
\newtheorem{lemma}{Lemma}
\newtheorem{assumption}{Assumption}

\theoremstyle{definition}

\theoremstyle{remark}

\newtheorem{example}{Example}

\allowdisplaybreaks[4]

\begin{document}

\title{On empirical Hodge Laplacians under the manifold hypothesis}
\author{Jan-Paul Lerch\thanks{Universit\"{a}t Bielefeld, Germany.} \qquad
Martin Wahl\footnotemark[1] \qquad Petr Zamolodtchikov\footnotemark[1] 
}
\date{}


\date{}
\maketitle
\begin{abstract}
Given i.i.d.~observations uniformly distributed on a closed submanifold of the Euclidean space, we study higher-order generalizations of graph Laplacians, so-called Hodge Laplacians on graphs, as approximations of the Laplace-Beltrami operator on differential forms. Our main result is a high-probability error bound for the associated Dirichlet forms. This bound improves existing Dirichlet form error bounds for graph Laplacians in the context of Laplacian Eigenmaps, and it provides insights into the Betti numbers studied in topological data analysis and the complementing positive part of the spectrum.
\end{abstract}


\section{Introduction}
Methods of dimensionality reduction uncover hidden information from complex data sets and high-dimensional observations. Leading examples are principal component analysis and its nonlinear extensions to kernel principal component analysis or manifold learning. Due to the availability of large amount of data, such methods have become indispensable tools throughout science and engineering. 

Principal component analysis is a basic linear dimensionality reduction method, in which the data is projected onto the linear space spanned by the leading eigenvectors of the empirical covariance matrix \cite{Joll02}. This allows to reduce the dimension, while preserving as much variation in the data as possible. Despite being a classical topic, principal component analysis is still intensively studied and exhibits many different phenomena in high dimensions \cite{JP18,Wain19,Kolt21,MR4517351}.

In contrast, Laplacian Eigenmaps and Diffusion Maps are instances of nonlinear dimensionality reduction. They are typically used under the so-called manifold hypothesis, where the data is assumed to be sampled from a low-dimensional submanifold in a high-dimensional Euclidean space \cite{Belkin03,Coifman06}. They are based on different graph Laplacians (unnormalized graph Laplacians, random walk graph Laplacian, etc.) and their spectral characteristics, which carry important information about the geometry of the underlying graph \cite{Chung97}. The study of the spectral properties of graph Laplacians as approximations of Laplace-Beltrami operators was initiated in \cite{Belkin06} and has since been explored using various approaches \cite{NIPS2006_5848ad95,Gine06,Burago14,MR4130541,MR4452681}, including connections to kernel principal component analysis \cite{Wahl24}.

Higher-order Laplacians are studied in the context of Hodge theory. Classical Hodge theory on Riemannian manifolds is defined in terms of the de Rham complex of differential forms on smooth manifolds and leads to the Laplace-Beltrami operator on differential forms. Analyzing the spectrum of the Laplace-Beltrami operator on $\ell$-forms \cite{Ros97}, particularly its null space, reveals fundamental topological information. For instance, the multiplicity of the zero eigenvalue corresponds to the $\ell$-th Betti number by Hodge's theorem. These concepts have been extended in various directions including simplicial complexes \cite{Eckmann44,Dod76}, metric measure spaces \cite{Smale12,HK24} and weighted graphs \cite{Lim20}. A relationship between random walks on simplicial complexes and higher-order (combinatorial) Laplacians has been established in \cite{Mukherjee16,MR3607124}.

In a complementary but related line of research, topological data analysis aims to provide statistical and algorithmic methods to understand the topological structure of data \cite{BCY18}. One of its most prominent techniques is persistent homology \cite{Edelsbrunner02} and their associated persistent Betti numbers, an extension of classical Betti numbers designed to capture topological structures that persist across scales. Significant statistical work has been conducted on these in a topological context \cite{Bubenik07} and in the context of generic chain complexes \cite{Ginzburg17}. Notably, both persistent homology and Hodge theory can be formulated algebraically as spectral problems \cite{Usher16}.

In this paper, we deal with the statistical analysis of data supported on a submanifold in a high-dimensional Euclidean space and consider the problem of approximating the Laplace-Beltrami operator on differential forms by appropriate empirical Hodge Laplacians. Inspired and guided by results in \cite{Lim11,Smale12,Lim20,HK24}, we construct an empirical exterior calculus, empirical $\ell$-forms, and an empirical Hodge Laplacian. Building on this, we turn to the statistical analysis of such empirical Hodge-Laplacians under the manifold hypothesis, and establish a non-asymptotic error bound for the associated empirical Dirichlet form. This upper bound provides a first step towards more sophisticated spectral convergence and approximation results, including the analysis of the Betti numbers studied in topological data analysis and the complementing positive part of the spectrum. Moreover, specialized to the empirical graph Laplacian, we improve existing Dirichlet form convergence rates in the context of Laplacian Eigenmaps and Diffusion Maps \cite{MR4130541,MR4491976}. In the proof, we combine tools from exterior calculus, matrix analysis, geometric analysis on Riemannian manifolds \cite{Gri09}, and the theory of U-statistics \cite{Mins24}.  

The paper is organized as follows. In \Cref{section:HodgeLaplacian} we provide a brief introduction to weighted Hodge Laplacians on a graph, followed by a discussion of empirical $\ell$-forms in \Cref{section:empiricalForms}. The Laplace--Beltrami operator in the context of Riemannian manifolds is discussed in \Cref{section:LapMfld} laying the fundament for the formulation of the main result in \Cref{thm:main:result} stated in \Cref{section:mainResult}. Sections \ref{section:bias} and \ref{section:variance} are dedicated to the proof of \Cref{thm:main:result}, and Section \ref{section:experiments} implements numerical experiments.

\subsection*{Basic notation}
For a natural number $n\geq 1$ the symmetric group on $n$ elements is denoted by $S_n$. $I_n$ denotes the $n\times n$ identity matrix. For a subset $J\subseteq \{1,\dots,n\}$, we denote by $J^{\complement}$ the complement of $J$ in $\{1,\dots,n\}$. For $q\geq 1$ and a real-valued random variable $X$ on a probability space $(\Omega,\mathcal{F},\mathbb{P})$ we write $\|X\|_{L^q}=\mathbb{E}^{1/q}|X|^q$ for the $L^q$ norm. Similarly, for $q=\infty$ we write $\|X\|_{L^\infty}$ for the (essential) supremum norm. Throughout the paper, $C>0$ denotes a constant that may change from line to line (by a numerical value).

\section{Hodge Laplacians on graphs}
\label{section:HodgeLaplacian}
Hodge Laplacians on graphs are higher-order generalizations of graph Laplacians. They can be interpreted as discrete analogous of Hodge theory on Riemannian manifolds, and they have been first introduced in the context of simplicial complexes. In this section, we summarize some basic elements and formulas of this theory in a form suitable for our study. Similar treatments can be found in \cite{Smale12,Lim20}.

Let $V=\{X_1,\dots,X_n\}$ be a finite set of data points (in a Euclidean space). We call a function $\boldsymbol{\omega}:V^{\ell+1}\rightarrow \mathbb{R}$ an $\ell$-form if it is alternating, that is 
\begin{align*}
    \boldsymbol{\omega}(X_{i_{\sigma(0)}},\dots,X_{i_{\sigma(\ell)}})=\operatorname{sgn}(\sigma)\boldsymbol{\omega}(X_{i_{0}},\dots,X_{i_{\ell}})
\end{align*}
for all $i_0,\dots,i_{\ell}\in\{1,\dots,n\}$ and all $\sigma\in S_{\ell+1}$. Given positive and symmetric weights $(k_{i_0\cdots i_{\ell}})$, we denote by $L^2_{\wedge}(V^{\ell+1})$ the Hilbert space of all $\ell$-forms endowed with the inner product 
\begin{align} \label{eq:def:discreteInnerProd}
    \langle \boldsymbol{\omega},\boldsymbol{\eta}\rangle_n= \tfrac{1}{(\ell +1)!}\sum_{i_0,\dots,i_\ell=1}^nk_{i_0\cdots i_\ell} \boldsymbol{\omega}(X_{i_0},\dots, X_{i_\ell})\boldsymbol{\eta}(X_{i_0},\dots, X_{i_\ell}).
\end{align}
Using the alternating property and the symmetry of the weights, we can also write 
\begin{align*} 
\langle \boldsymbol{\omega},\boldsymbol{\eta}\rangle_n= \sum_{1\leq i_0<\dots<i_\ell\leq n}k_{i_0\cdots i_\ell} \boldsymbol{\omega}(X_{i_0},\dots, X_{i_\ell})\boldsymbol{\eta}(X_{i_0},\dots, X_{i_\ell}).
\end{align*}
Note that the results of this section are also true if the weights $(k_{i_0\cdots i_{\ell}})$ are non-negative. In this case, we require the additional property that $k_{i_0\cdots i_{\ell}}\neq 0$ implies that $k_{i_1\cdots i_{\ell}}\neq 0$ for all $i_0,\ldots, i_{\ell}\in\{1,\dots,n\}$ and all $\ell\geq 1$ (often called downward-closed property), and $L^2_{\wedge}(V^{\ell+1})$ is understood as the Hilbert space of functions that are zero for $\ell$-tuples $(X_{i_0},\dots, X_{i_\ell})$ such that $k_{i_0\cdots i_{\ell}}=0$. Moreover, we introduce the $\ell$-coboundary operator $\delta_\ell:L^2_{\wedge}(V^{\ell+1})\rightarrow L^2_{\wedge}(V^{\ell+2})$ defined by
        \begin{align*}
            (\delta_\ell\boldsymbol{\omega}) (X_{i_0},\dots,X_{i_{\ell+1}})=\sum_{j=0}^{\ell+1}(-1)^j\boldsymbol{\omega}(X_{i_0},\dots,\widehat{X_{i_j}},\dots,X_{i_{\ell+1}}),    
        \end{align*}
where $\widehat{X_{i_j}}$ means that $X_{i_j}$ is omitted. The above information can be summarized in the cochain complex
\begin{equation} \label{eq:discreteComplex}
  \xymatrix{
0 \ar[r] & L^2(V) \ar[r]^{\delta_0} & L^2_{\wedge}(V^{2}) \ar[r]^-{\delta_1} & \cdots \ar[r]^-{\delta_{\ell-1}} & L^2_{\wedge}(V^{\ell+1})  \ar[r]^-{\delta_{\ell}} & \cdots
}  
\end{equation}
satisfying $\delta_\ell\circ \delta_{\ell-1}=0$ for every $\ell\geq 1$ (see Theorem 5.7 in \cite{Lim20}). If $\ell$ is clear from the context, we will also omit the subscript and write $\delta$ instead of $\delta_\ell$. For a function $\mathbf{f}\in L^2(V)$, we will e.g.~often abbreviate $\delta_0 \mathbf{f}$ to $\delta \mathbf{f}$. Let $\delta_{\ell}^*:L^2_{\wedge}(V^{\ell+2})\rightarrow L^2_{\wedge}(V^{\ell+1})$ be the adjoint of $\delta_{\ell}$ defined by the identity $\langle \delta_{\ell}^*\boldsymbol{\omega},\boldsymbol{\eta}\rangle_n=\langle \boldsymbol{\omega},\delta_{\ell}\boldsymbol{\eta}\rangle_n$, valid for all $(\ell+1)$-forms $\boldsymbol{\omega}$ and all $\ell$-forms $\boldsymbol{\eta}$. Explicitly, an elementary computation leads to (compare to \cite{Smale12})
\begin{align}\label{eq:adjoint:delta}
    (\delta_{\ell}^*\boldsymbol{\omega})(X_{i_0},\dots,X_{i_\ell})=\sum_{j=1}^n \frac{k_{ji_0\cdots i_\ell}}{k_{i_0\cdots i_\ell}}\boldsymbol{\omega}(X_{j},X_{i_0},\dots, X_{i_\ell}).
\end{align}
Finally, we define the up and down Hodge Laplacian by 
    \begin{align*}
         \mathscr{L}_\ell^{\up}=\delta_{\ell}^*\delta_{\ell},\qquad \mathscr{L}_\ell^{\down}=\delta_{\ell-1}\delta_{\ell-1}^* 
    \end{align*}
    for $ \ell\geq 1$, as well as the full Hodge Laplacian by $$\mathscr{L}_0=\mathscr{L}_0^{\up}=\delta_0^*\delta_0$$ and, for $\ell\geq 1$,
    \begin{align*}
    \mathscr{L}_\ell=\mathscr{L}_\ell^{\up}+\mathscr{L}_\ell^{\down}=\delta_{\ell}^*\delta_{\ell}+\delta_{\ell-1}\delta_{\ell-1}^*.
    \end{align*}
    \begin{example}[Graph Laplacians]
        \label{ex.graph.laplacian}
        Consider the case $\ell=0$. Suppose that $K=(k_{ij})\in\mathbb{R}^{n\times n}$ is a symmetric matrix with non-negative entries such that the so-called degree matrix $D=\operatorname{diag}(d_1,\dots,d_n)$, $d_i=\sum_{j=1}^nk_{ij}$ is non-singular and such that the downward-closed property holds. Then 
        \begin{align*}
            (\delta_0\mathbf{f})(X_i,X_j)=\mathbf{f}(X_j)-\mathbf{f}(X_i),\qquad \mathbf{f}\in L^2(V),
        \end{align*}
        is a discrete version of the gradient and $\delta_0^*$ is given by
        \begin{align*}
            (\delta_0^*\boldsymbol{\omega})(X_i)=\sum_{j=1}^n\frac{k_{ij}}{k_i}\boldsymbol{\omega}(X_j,X_i),\qquad \boldsymbol{\omega}\in L_\wedge^2(V^2).
        \end{align*}
        Hence, if $k_1=\dots=k_n = 1$, then
        \begin{align*}
            (\mathscr{L}_0\mathbf{f})(X_i)=\sum_{j=1}^nk_{ij}(\mathbf{f}(X_i)-\mathbf{f}(X_j)),
        \end{align*}
        meaning that $\mathscr{L}_0=D-K$ if we identify $\mathbf{f}\in L^2(V)$ with the vector $(\mathbf{f}(X_1),\dots,\mathbf{f}(X_n))^\top\in\mathbb{R}^n$ and $\mathscr{L}_0$ with the associated matrix representation in $\mathbb{R}^{n\times n}$. Moreover, if $k_i=d_i$ for all $i=1,\dots,n$, then 
        \begin{align*}
            (\mathscr{L}_0\mathbf{f})(X_i)=\sum_{j=1}^n\frac{k_{ij}}{k_i}(\mathbf{f}(X_i)-\mathbf{f}(X_j)),
        \end{align*}
        meaning that $\mathscr{L}_0=I_n-D^{-1}K$ with the identification above. As a result, in these two cases, $\mathscr{L}_0$ coincides with the unnormalized graph Laplacian and the random walk graph Laplacian, respectively \cite{Chung97,vLux07}.
    \end{example}

    \begin{example}[Graph Helmholtzian] 
        \label{ex.graph.helmholtzian}
        Inverting the order in which we compose $\delta_0$ and $\delta_0^*$ in Example \ref{ex.graph.laplacian} leads to 
        \begin{align*}
            (\mathscr{L}_{1}^{\down}\boldsymbol{\omega})(X_i, X_j) &= \sum_{m=1}^n \frac{k_{mj}}{k_{j}}\boldsymbol{\omega}(X_m, X_{j}) - \sum_{m=1}^n \frac{k_{mi}}{k_{i}}\boldsymbol{\omega}(X_m, X_{i})
        \end{align*}
        for $\boldsymbol{\omega} \in L^2_{\wedge}(V^2)$. The above is identical to \cite[Eqn.~(2.7)]{eidi2023irreducibility}. Moreover, the coboundary operator $\delta_1$ is given by
        \begin{align}\label{eq:Helm:up}
            (\delta_1\boldsymbol{\omega})(X_{i}, X_{j}, X_{k}) = \boldsymbol{\omega}(X_{j}, X_{k}) - \boldsymbol{\omega}(X_{i}, X_{k}) + \boldsymbol{\omega}(X_{i}, X_{j})
        \end{align}
        for $\boldsymbol{\omega} \in L^2_{\wedge}(V^2)$.  Combining this with the formula \eqref{eq:adjoint:delta} for its adjoint, we get
        \begin{align*}
            (\mathscr{L}_1^{\up}\boldsymbol{\omega})(X_{i}, X_{j}) &= \sum_{m =1}^n\frac{k_{m i j}}{k_{i j}} \big(\boldsymbol{\omega}(X_{i}, X_{j}) - \boldsymbol{\omega}(X_m, X_{j}) + \boldsymbol{\omega}(X_m, X_{i})\big),
        \end{align*}
        which corresponds to \cite[Eqn.\ (2.6)]{eidi2023irreducibility} and the up Hodge Laplacian identity on weighted simplicial complexes in \cite{Horak13}.  Depending on the choice of the weights, the up and down Hodge Laplacians are connected to a random walk on the weighted simplicial complex. For instance, if $k_{ij} = \sum_{m =1}^n k_{mij}$ for all $i\neq j$, then 
         \begin{align*}
        \frac{1}{2}((I - \mathscr{L}_1^{\up})\boldsymbol{\omega})(X_i, X_j) = \sum_{m = 1}^{n}\frac{k_{mij}}{k_{ij}}\bigg(\frac 12 \boldsymbol{\omega}(X_m, X_j) + \frac 12\boldsymbol{\omega}(X_i, X_m)\bigg),
    \end{align*}
    so $(I - \mathscr{L}_1^{\up})/2$ can be interpreted as a transition operator of a Markov chain \cite{MR3607124}, \cite{Mukherjee16}.
    \end{example}    
    The operators $\mathscr{L}_\ell$, $\ell\geq 0$ are by construction self-adjoint and positive semi-definite and thus have real and non-negative eigenvalues. 
    These eigenvalues contain topological information about the underlying graph. For instance, it is well-known that the multiplicities of the eigenvalue zero of the unnormalized graph Laplacian $D-K$ (that is $\dim(\ker(D-K))$) from the above example is equal to the number of connected components of the weighted graph $(V,K)$. 
    Note that  $\dim(\ker(D-K))=1$ if all weights are non-zero, while it might be strictly larger for non-negative weights. 
    Moreover, the first nonzero eigenvalue is related to the Cheeger constant and satisfies the so-called Cheeger inequality. For more details see \cite{Chung97,Lev+17}. 
    Similar statements for $\ell\geq 1$ are encoded in the so-called Hodge decomposition, which can be deduced from $\delta_\ell\circ \delta_{\ell-1}=0$ and results from linear algebra in our finite-dimensional setting (see Section~4.3 in \cite{Lim20}). 
    First, $\ker(\mathscr{L}_\ell)=\ker(\delta_\ell)\cap \ker(\delta_{\ell-1}^*)$ and thus $\operatorname{im}(\mathscr{L}_\ell)=\operatorname{im}(\delta_\ell^*)\oplus\operatorname{im}(\delta_{\ell-1})$. 
    In particular, we have 
    \begin{align*}
        L^2_\wedge(V^{\ell+1})=\overunderbraces{&&\br{2}{\ker(\delta_{\ell})}}
        {&\operatorname{im}(\delta_{\ell}^*) \oplus &  \ker(\mathscr{L}_\ell) &\oplus\operatorname{im}(\delta_{\ell-1})}
        {&\br{2}{\ker(\delta_{\ell-1}^*)}}.
    \end{align*}
    Thus the $\ell$-th cohomology group $\ker(\delta_\ell)/\operatorname{im}(\delta_{\ell-1})$ is isomorphic to $\ker(\mathscr{L}_\ell)$. The quantity $\dim(\ker(\mathscr{L}_\ell))$ is also called the $\ell$-th Betti number. 
    Moreover, the set of nonzero eigenvalues (counted with multiplicities) of $\mathscr{L}_\ell$ is equal to the union of the nonzero eigenvalues of $\mathscr{L}_\ell^{\up}$ and the nonzero eigenvalues of $\mathscr{L}_\ell^{\down}$. 
    Since the nonzero eigenvalues of $\mathscr{L}_\ell^{\down}$ are equal to the nonzero eigenvalues of $\mathscr{L}_{\ell-1}^{\up}$, it suffices to focus on $(\mathscr{L}_\ell^{\up})_{\ell\geq 0}$ when studying the eigenvalues of all $(\mathscr{L}_\ell)_{\ell\geq 0}$. 
    By the min-max characterization of eigenvalues, it is thus an important first step to study the quadratic form $\langle\boldsymbol{\omega},\mathscr{L}_{\ell-1}^{\up}\boldsymbol{\omega}\rangle_n$, which will be the main focus of this paper.

    \section{Empirical $\ell$-forms}
    \label{section:empiricalForms}
    In this section, we endow the spaces $L^2_\wedge(V^{\ell+1})$, $\ell\geq 0$ with an additional wedge product $\wedge$. This will strengthen the analogy to differential $\ell$-forms on manifolds, and it will allow us to define certain $\ell$-forms that are characterized by functions only. Similar results can be found in \cite{HK24}, where a tensor product formulation for general non-local differential complexes was introduced. Classical background may for instance be found in~\cite{War83,MT02}. 
    
    For $\boldsymbol{\omega}\in L^2_\wedge(V^{\ell+1})$ and $\boldsymbol{\eta}\in L^2_\wedge(V^{m+1})$, we define $\boldsymbol{\omega}\wedge \boldsymbol{\eta}\in L^2_\wedge(V^{\ell+m+1})$ by 
    \begin{align} \label{eq:def:wedgeProduct}
        \begin{split}
        &(\boldsymbol{\omega}\wedge \boldsymbol{\eta})(X_{i_0},\dots, X_{i_{\ell+m}})\\
        &=\frac{1}{(\ell+m+1)!}\sum_{\sigma\in S_{\ell+m+1}}\operatorname{sgn}(\sigma)\boldsymbol{\omega}(X_{i_{\sigma(0)}},\dots, X_{i_{\sigma(\ell)}})\boldsymbol{\eta}(X_{i_{\sigma(\ell)}},\dots, X_{i_{\sigma(\ell+m)}}).
        \end{split}
    \end{align}
    In what follows, we collect some basic properties of the wedge product.

    \begin{lemma} If $\mathbf{f}$ is a $0$-form and $\boldsymbol{\omega}$ is an $\ell$-form, then 
    \begin{align*}
        (\mathbf{f} \boldsymbol{\omega})(X_{i_0},\dots, X_{i_{\ell}}):=(\mathbf{f}\wedge \boldsymbol{\omega})(X_{i_0},\dots, X_{i_{\ell}})=\frac{f(X_{i_0})+\dots +f(X_{i_\ell})}{\ell+1}\boldsymbol{\omega}(X_{i_0},\dots, X_{i_{\ell}}).
    \end{align*}
    \end{lemma}

    \begin{proof}
        For each $a=0,\dots,\ell$, there are $\ell!$ permutations $\sigma$ on $\{0,\dots,\ell\}$ with $\sigma(0)=a$. Hence,
        \begin{align*}
            (\mathbf{f}\wedge \boldsymbol{\omega})(X_{i_0},\dots, X_{i_{\ell}})&=\frac{1}{(\ell+1)!}\sum_{\sigma\in S_{\ell+1}}\mathbf{f}(X_{i_{\sigma(0)}})\boldsymbol{\omega}(X_{i_{0}},\dots, X_{i_{\ell}})\\
            &=\Big(\frac{1}{\ell+1}\sum_{a=0}^\ell \mathbf{f}(X_{i_a})\Big)\boldsymbol{\omega}(X_{i_{0}},\dots, X_{i_{\ell}}),
        \end{align*}
        where we used the alternating property in the first equality.
    \end{proof}

    A variant of the following lemma has also been shown in Proposition 3.2 in \cite{HK24}. Here we give a slightly different argument based on the Leibniz rule for the wedge product.

    \begin{lemma}\label{lemma:derivative:empirical:l:form}
        Let $\mathbf{f}_1,\dots,\mathbf{f}_\ell\in L^2(V)$ and $\boldsymbol{\omega}=\mathbf{f}_1(\delta_0\mathbf{f}_2\wedge\dots\wedge \delta_0\mathbf{f}_\ell)\in L_\wedge^2(V^\ell)$. Then we have 
        \[ \delta_{\ell-1}\boldsymbol{\omega}=\delta_{\ell-1} \left( \mathbf{f}_1(\delta_0\mathbf{f}_2\wedge\dots\wedge \delta_0\mathbf{f}_\ell)\right) =\delta_0\mathbf{f}_1\wedge\dots\wedge \delta_0\mathbf{f}_\ell.\]
    \end{lemma}

    \begin{proof}
        Set $\boldsymbol{\eta}=\delta_0\mathbf{f}_2\wedge\dots\wedge \delta_0\mathbf{f}_\ell$. Then we have $\boldsymbol{\omega}=\mathbf{f}_1\wedge \boldsymbol{\eta}$ and by definition 
        \begin{align}\label{eq:lhs:derivative:form}
            &(\delta_{\ell-1}\boldsymbol{\omega})(X_{i_0},\dots,X_{i_\ell})\\
            &=\sum_{a=0}^{\ell}(-1)^a(\mathbf{f}_1\wedge \boldsymbol{\eta})(X_{i_0},\dots,\widehat{X_{i_a}},\dots,X_{i_{\ell}})\nonumber\\
            &=\frac{1}{\ell!}\sum_{a=0}^\ell\sum_{\substack{\sigma\in S_{\ell+1}\\ \sigma(a)=a}}(-1)^a \operatorname{sgn}(\sigma)\mathbf{f}_1(X_{i_{\sigma(0)}})\boldsymbol{\eta}(X_{i_{\sigma(0)}},\dots,\widehat{X_{i_{\sigma(a)}}},\dots,X_{i_{\sigma(\ell)}}),\nonumber
        \end{align}
        where the second equality follows from the definition of $\wedge$ and the fact that all permutations $\sigma\in S_{\ell+1}$ with $\sigma(a)=a$ are in bijection to all permutations on $\{0,\dots,\ell\}\setminus \{a\}$.  
        On the other hand, we have
        \begin{align*}
            &(\delta_0\mathbf{f}_1\wedge \boldsymbol{\eta}+\mathbf{f}_1\delta_{\ell-1}\boldsymbol{\eta})(X_{i_{0}},\dots, X_{i_{\ell}})\\
            &=\frac{1}{(\ell+1)!}\sum_{\sigma\in S_{\ell+1}}\operatorname{sgn}(\sigma)(\mathbf{f}_1(X_{i_{\sigma(1)}})-\mathbf{f}_1(X_{i_{\sigma(0)}}))\boldsymbol{\eta}(X_{i_{\sigma(1)}},\dots, X_{i_{\sigma(\ell)}})\\
            &+\frac{1}{(\ell+1)!}\sum_{\sigma\in S_{\ell+1}}\operatorname{sgn}(\sigma)\mathbf{f}_1(X_{i_{\sigma(0)}})\sum_{a=0}^{\ell}(-1)^a\boldsymbol{\eta}(X_{i_{\sigma(0)}},\dots,\widehat{X_{i_{\sigma(a)}}},\dots,X_{i_{\sigma(\ell)}})\\
            &=\frac{1}{(\ell+1)!}\sum_{\sigma\in S_{\ell+1}}\operatorname{sgn}(\sigma)\mathbf{f}_1(X_{i_{\sigma(1)}})\boldsymbol{\eta}(X_{i_{\sigma(1)}},\dots, X_{i_{\sigma(\ell)}})\\
            &+\frac{1}{(\ell+1)!}\sum_{\sigma\in S_{\ell+1}}\operatorname{sgn}(\sigma)\mathbf{f}_1(X_{i_{\sigma(0)}})\sum_{a=1}^{\ell}(-1)^a\boldsymbol{\eta}(X_{i_{\sigma(0)}},\dots,\widehat{X_{i_{\sigma(a)}}},\dots,X_{i_{\sigma(\ell)}}).
        \end{align*}
        In the first equality we transform the second term $\mathbf{f}_1\delta_{\ell-1}\boldsymbol{\eta}$, making use of the alternating property. For the second equality we observe that the terms for $a=0$ in the second sum cancel with the negative part of the first sum. Substituting $\sigma$ by $\sigma\circ (0,1,\dots,a)$, we arrive at
        \begin{align}
            &(\delta_0\mathbf{f}_1\wedge \boldsymbol{\eta}+\mathbf{f}_1\delta_{\ell-1}\boldsymbol{\eta})(X_{i_{0}},\dots, X_{i_{\ell}})\label{eq:rhs:derivative:form}\\
            &=(\ell+1)\frac{1}{(\ell+1)!}\sum_{\sigma\in S_{\ell+1}}\operatorname{sgn}(\sigma)\mathbf{f}_1(X_{i_{\sigma(1)}})\boldsymbol{\eta}(X_{i_{\sigma(1)}},\dots, X_{i_{\sigma(\ell)}})\nonumber\\
            &=\frac{1}{\ell!}\sum_{\substack{\sigma\in S_{\ell+1}\\\sigma(0)=0}}\operatorname{sgn}(\sigma)\mathbf{f}_1(X_{i_{\sigma(1)}})\boldsymbol{\eta}(X_{i_{\sigma(1)}},\dots, X_{i_{\sigma(\ell)}})\nonumber\\
            &+\frac{1}{\ell!}\sum_{a=1}^\ell\sum_{\substack{\sigma\in S_{\ell+1}\\\sigma(0)=a}}\operatorname{sgn}(\sigma)\mathbf{f}_1(X_{i_{\sigma(1)}})\boldsymbol{\eta}(X_{i_{\sigma(1)}},\dots, X_{i_{\sigma(\ell)}})\nonumber\\
            &=\frac{1}{\ell!}\sum_{\substack{\sigma\in S_{\ell+1}\\\sigma(0)=0}}\operatorname{sgn}(\sigma)\mathbf{f}_1(X_{i_{\sigma(1)}})\boldsymbol{\eta}(X_{i_{\sigma(1)}},\dots, X_{i_{\sigma(\ell)}})\nonumber\\
            &+\frac{1}{\ell!}\sum_{a=1}^\ell\sum_{\substack{\sigma\in S_{\ell+1}\\\sigma(0)=a}}(-1)^a\operatorname{sgn}(\sigma)\mathbf{f}_1(X_{i_{\sigma(0)}})\boldsymbol{\eta}(X_{i_{\sigma(0)}},\dots,\widehat{X_{i_{\sigma(a)}}},\dots,X_{i_{\sigma(\ell)}}),\nonumber
        \end{align}
        where the last equality follows from substituting again $\sigma$ by $\sigma\circ (0,1,\dots,a)$. Combining \eqref{eq:lhs:derivative:form} and \eqref{eq:rhs:derivative:form}, we arrive at the Leibniz rule \cite{War83} 
        \begin{align*}
            \delta_{\ell-1}\boldsymbol{\omega}=\delta_0\mathbf{f}_1\wedge \boldsymbol{\eta}&+\mathbf{f}_1\delta_{\ell-1}\boldsymbol{\eta}, 
        \end{align*}
        from which the claim follows by induction. Indeed, for $\ell=2$ we have $\delta_1(\mathbf{f}_1 \delta_0\mathbf{f}_2)=\delta_0\mathbf{f}_1\wedge \delta_0\mathbf{f}_2$ since $\delta_1\circ \delta_0=0$, and in the induction step we use $\delta_{\ell-1}\boldsymbol{\eta}=0$ since $\delta_{\ell-1}\circ \delta_{\ell-2}=0$. This completes the proof.
    \end{proof}
    The following lemma is a variant of formula (11) in \cite{HK24}.

    \begin{lemma}\label{lemma:derivative:empirical:l:form:det}
        If $\mathbf{f}_1,\dots,\mathbf{f}_\ell\in L^2(V)$, then 
        \begin{align*}
            (\delta_0\mathbf{f}_1\wedge\dots\wedge \delta_0\mathbf{f}_\ell)(X_{i_{0}},\dots, X_{i_{\ell}})=\frac{1}{\ell!}\det_{\ell\times\ell}\big( \delta_0 \mathbf{f}_a(X_{i_0},X_{i_b})\big)_{a,b}.
        \end{align*}
    \end{lemma}

    \begin{proof}
        First note that the right-hand side $\det( \delta_0 \mathbf{f}_a(X_{i_0},X_{i_b}))$ is in $L^2_\wedge(V^{\ell+1})$, as can be seen by the multilinearity of the determinant. Hence, 
        \begin{align}\label{eq:D:alternating}
            \det( \delta_0 \mathbf{f}_a(X_{i_0},X_{i_b}))=\frac{1}{(\ell+1)!}\sum_{\sigma\in S_{\ell+1}}\operatorname{sgn}(\sigma)\det( \delta_0 \mathbf{f}_{a}(X_{i_{\sigma(0)}},X_{i_{\sigma(b)}})).
        \end{align}
        On the other hand, for $\boldsymbol{\omega}=\delta_0\mathbf{f}_1$ and $\boldsymbol{\eta}=\delta_0\mathbf{f}_2\wedge\dots\wedge \delta_0\mathbf{f}_\ell$ it holds that, for each $b=1,\dots\ell$,
        \begin{align}\label{eq:wedge:product:2:l-1}
        &(\boldsymbol{\omega}\wedge \boldsymbol{\eta})(X_{i_0},\dots,X_{i_\ell})\\
        &=\frac{1}{(\ell+1)!}\sum_{\sigma\in S_{\ell+1}}\operatorname{sgn}(\sigma)\boldsymbol{\omega}(X_{i_{\sigma(0)}}, X_{i_{\sigma(1)}})\boldsymbol{\eta}(X_{i_{\sigma(1)}},\dots, X_{i_{\sigma(\ell)}})\nonumber\\
        &=\frac{1}{(\ell+1)!}\sum_{\sigma\in S_{\ell+1}}(-1)^{b+1}\operatorname{sgn}(\sigma)\boldsymbol{\omega}(X_{i_{\sigma(0)}}, X_{i_{\sigma(b)}})\boldsymbol{\eta}(X_{i_{\sigma(0)}},\dots,\widehat{X_{i_{\sigma(b)}}},\dots,X_{i_{\sigma(\ell)}})\nonumber,
        \end{align}
        where the first equality is by \eqref{eq:def:wedgeProduct} and the second equality follows from substituting $\sigma$ by $\sigma\circ (b,b-1,\dots,0)$. Using these two properties, we prove the claim by induction on $\ell$. For $\ell =1$ the claim is clear. Let us assume the claim holds for $\ell-1\geq 1$. Then we have
        \begin{align*}
            \det( \delta_0 \mathbf{f}_a(X_{i_0},X_{i_b})) &=\frac{1}{(\ell+1)!}\sum_{\sigma\in S_{\ell+1}}\operatorname{sgn}(\sigma)\det( \delta_0 \mathbf{f}_{a}(X_{i_{\sigma(0)}},X_{i_{\sigma(b)}}))\\
            &=\frac{1}{(\ell+1)\ell}\sum_{\sigma\in S_{\ell+1}}\sum_{b=1}^\ell(-1)^{b+1}\operatorname{sgn}(\sigma)\delta_0\mathbf{f}_{1}(X_{i_{\sigma(0)}},X_{i_{\sigma(b)}})\\
            &\qquad\qquad\qquad\qquad\qquad\cdot(\delta_0\mathbf{f}_2\wedge\dots\wedge \delta_0\mathbf{f}_\ell)(X_{i_{\sigma(0)}},\dots,\widehat{X_{i_{\sigma(b)}}},\dots,X_{i_{\sigma(\ell)}})\\
            &=\frac{(\ell+1)!\ell}{(\ell+1)\ell}(\delta_0\mathbf{f}_1\wedge\dots\wedge \delta_0\mathbf{f}_\ell)(X_{i_0},\dots,X_{i_\ell}),
        \end{align*}
       where we used \eqref{eq:D:alternating} in the first equality, the induction hypothesis and the Laplace expansion in the second equality, and \eqref{eq:wedge:product:2:l-1} in the last equality.
    \end{proof}

    \begin{lemma}\label{lemma:representation:Dirichlet:form:empirical:l:form}
    Let $\mathbf{f}_1,\dots,\mathbf{f}_\ell\in L^2(V)$ and $\boldsymbol{\omega}=\mathbf{f_1}\cdot(\delta\mathbf{f}_2\wedge\dots\wedge \delta\mathbf{f_\ell})\in L^2_\wedge(V^{\ell})$. Then 
    \begin{align*}
        \langle\boldsymbol{\omega},\mathscr{L}_{\ell-1}^{\up}\boldsymbol{\omega}\rangle_n=\frac{1}{\ell!^2}\sum_{1\leq i_0<\dots<i_\ell\leq n}k_{i_0\cdots i_\ell}\det_{\ell\times\ell}\big( \delta\mathbf{f}_a(X_{i_0},X_{i_b})\big)^2.
    \end{align*}
    \end{lemma}

    \begin{proof}
        By definition of the up-Hodge Laplacian and Lemma \ref{lemma:derivative:empirical:l:form}, we have
        \begin{align*}
            \langle\boldsymbol{\omega},\mathscr{L}_{\ell-1}^{\up}\boldsymbol{\omega}\rangle_n=\langle\boldsymbol{\omega},\delta_{\ell-1}^*\delta_{\ell-1}\boldsymbol{\omega}\rangle_n&=\langle\delta_{\ell-1}\boldsymbol{\omega},\delta_{\ell-1}\boldsymbol{\omega}\rangle_n=\langle\delta_0\mathbf{f}_1\wedge\dots\wedge \delta_0\mathbf{f}_\ell,\delta_0\mathbf{f}_1\wedge\dots\wedge \delta_0\mathbf{f}_\ell\rangle_n.
        \end{align*}
        Hence, the claim follows from Lemma \ref{lemma:derivative:empirical:l:form:det} and the definition of the inner product in \eqref{eq:def:discreteInnerProd}.
    \end{proof}

    In what follows, we call $(\ell-1)$-forms of the form $\boldsymbol{\omega}=\mathbf{f_1}(\delta\mathbf{f}_2\wedge\dots\wedge \delta\mathbf{f_\ell})$ an empirical $(\ell-1)$-form.

    \begin{example}
        For $\ell=1$, Lemma \ref{lemma:representation:Dirichlet:form:empirical:l:form} reduces to the well-known identity
        \begin{align}\label{eq:GL:BF}
            \langle\mathbf{f},\mathscr{L}_{0}\mathbf{f}\rangle_n=\frac{1}{2}\sum_{i,j=1}^nk_{ij}(\mathbf{f}(X_j)-\mathbf{f}(X_i))^2.
        \end{align}
    \end{example}
    
\section{The Laplace-Beltrami operator on a manifold}
\label{section:LapMfld}

Let $\mathcal{M}$ be a $d$-dimensional submanifold of $\mathbb{R}^p$, equipped with the Riemannian metric induced by the ambient space. We assume that $\mathcal{M}$ is closed, connected, and that $\operatorname{vol}(\mathcal{M})=1$. Let $C^\infty(\mathcal{M})$ be the set of all smooth and real-valued functions on $\mathcal{M}$. Let $\Delta$ be the Laplace-Beltrami operator on $\mathcal{M}$ (with the sign convention that $\Delta$ is positive on $C^\infty(\mathcal{M})$ endowed with the $L^2$-inner product), and let $(e^{-t\Delta})_{t\geq 0}$ be the heat semigroup on $\mathcal{M}$. Since $\mathcal{M}$ is closed, $e^{-t\Delta}$ has an integral kernel $k_t$ (the so-called heat kernel) satisfying
\begin{align}\label{eq:integral:kernel}
    (e^{-t\Delta}f)(x)=\int_{\mathcal{M}}k_t(x,y)f(y)\,dy
\end{align}
for all $x\in\mathcal{M}$, all $t> 0$, and all $f\in C^\infty(\mathcal{M})$, where $dy$ denotes the volume measure induced by the Riemannian metric on $\mathcal{M}$. The heat kernel $k_t$ is symmetric, positive, and satisfies $\int_{\mathcal{M}}k_t(x,y)\,dy=1$ for all $x\in\mathcal{M}$. For more background see \cite{Ros97,Gri09,War83}.

For $\ell\geq 0$, let $\Omega^{\ell}(\mathcal{M})$ be the set of all smooth differential $\ell$-forms, let $d$ be the exterior differentiation operator, and let $\wedge$ be the wedge product. The Riemannian metric defines an inner product $\langle \cdot,\cdot\rangle_x$ on the cotangent space at point $x$, leading to a global inner product $\int_{\mathcal{M}}\langle \cdot,\cdot\rangle_x\,dx$ on $\Omega^{1}(\mathcal{M})$. Similarly, the inner product induces an inner product on $\Omega^{\ell}(\mathcal{M})$, which we denote by $\langle \cdot,\cdot\rangle$. In what follows it is important that if $f_1,\dots, f_{\ell}\in C^\infty(\mathcal{M})$, then 
\begin{align*}
        \langle df_1\wedge\dots\wedge df_\ell,df_1\wedge\dots\wedge df_\ell\rangle=\int_{\mathcal{M}}\det\begin{pmatrix}
                    \langle df_1,df_1\rangle_x & \ldots & \langle df_1,df_\ell\rangle_x \\
                    \vdots & \ddots & \vdots \\
                    \langle df_\ell,df_1\rangle_x & \ldots & \langle df_\ell,df_\ell\rangle_x
                \end{pmatrix}\,dx
    \end{align*}
    and the individual inner products $\langle df_a,df_b\rangle_x$ coincide with the carré du champ operator of $f_a$ and $f_b$. Again, the above information can be summarized into the de Rham cochain complex
        \[
        \xymatrix{
        0 \ar[r] & \Omega^0(\mathcal{M}) \ar[r]^-{d_0} & \Omega^1(\mathcal{M}) \ar[r]^-{d_1} & \cdots \ar[r]^-{d_{\ell-1}} & \Omega^\ell(\mathcal{M})  \ar[r]^-{d_{\ell}} & \cdots,
        }
        \]
which is the differential counterpart to \eqref{eq:discreteComplex}. Finally, for each $\ell\geq 0$, let $d_{\ell}^*$ be the adjoint of $d_\ell$ with respect to $\langle \cdot,\cdot\rangle$. Then the up and down Laplace-Beltrami operators on $\ell$-forms are
        \begin{align*}
            \Delta_\ell^{\up}=d_\ell^*d_{\ell},\qquad \Delta_\ell^{\down}=d_{\ell-1}d_{\ell-1}^*.
        \end{align*} 
and the Laplace-Beltrami operator on $\ell$-forms are $\Delta_0=\Delta_0^{\up}$ and
        \begin{align*}
            \Delta_\ell=\Delta_\ell^{\up}+\Delta_\ell^{\down}=d_\ell^*d_{\ell}+d_{\ell-1}d_{\ell-1}^*
        \end{align*} 
for $\ell \geq 1$. Similarly as in the graph theoretic setting, an Hodge decomposition holds stating that $\Omega^{\ell}(\mathcal{M})=\operatorname{im}(d_{\ell-1})\oplus \ker(\Delta_\ell)\oplus \operatorname{im}(d_\ell^*)$. Moreover, topological information is contained in $\ker(\Delta_\ell)$, which is isomorphic to the $\ell$th de Rham cohomology group. For more details, see \cite{Ros97}.

\section{Main result: Dirichlet form error bound} \label{section:mainResult}

\subsection{Assumptions and main result}

The main goal of this paper is to relate $\mathscr{L}_{\ell-1}^{\up}$ to $\Delta_{\ell-1}^{\up}$ in the case that $V=\{X_1,\dots,X_n\}$ is a sample of independent and identical distributed points in a submanifold of the Euclidean space. 

\begin{assumption}[Manifold hypothesis]\label{Ass:MH} Let $X_1,\dots,X_n$ be independent and identical distributed random variables uniformly distributed in a closed and connected submanifold $\mathcal{M}\subseteq\mathbb{R}^p$ with $\dim(\mathcal{M})=d$ and $\operatorname{vol}(\mathcal{M})=1$.
\end{assumption}

The manifold hypothesis is crucial in modern theory of machine learning \cite{MR2964193}. Under Assumption \ref{Ass:MH}, the empirical Hodge Laplacian can be analyzed as an approximation of the Laplace-Beltrami operator. In this paper, we make a first step in this direction and study the quadratic form $\langle\boldsymbol{\omega},\mathscr{L}_{\ell-1}^{\up}\boldsymbol{\omega}\rangle_n$ as approximations of the Dirichlet form or energy $\langle\omega,\Delta_{\ell-1}^{\up}\omega\rangle$ for certain (empirical) $\ell$-forms $\omega$ and $\boldsymbol{\omega}$. More precisely, for a fixed set $f_1,\dots,f_\ell\in C^\infty(\mathcal{M})$ and its restrictions $\mathbf{f}_1,\dots,\mathbf{f}_\ell\in L^2_\wedge(V)$ to the data points, we consider 
\begin{align}\label{eq:definition:emprical:l:form}
    \omega=f_1(df_2\wedge\dots\wedge df_\ell),\qquad\boldsymbol{\omega}=\mathbf{f_1}(\delta\mathbf{f}_2\wedge\dots\wedge \delta\mathbf{f_\ell})\in L^2_\wedge(V^{\ell})
\end{align}
As weights, we consider 
\begin{align}\label{eq:definition:weights}
    k_{i_0\dots i_\ell}=\frac{1}{\binom{n}{\ell+1}}\frac{\ell !}{(2t)^\ell}\Big(\frac{1}{\ell+1}\sum_{a=0}^\ell\prod_{\substack{b=0\\b\neq a}}^\ell k_t(X_{i_a},X_{i_b})\Big)
\end{align}
for all $i_0,\dots,i_\ell\in\{1,\dots,n\}$ and all $\ell\geq 0$, and with time parameter $t>0$. See also equation (42) in \cite{HK24} and equation (3) in \cite{Smale12}. Here, $k_t$ denoted the heat kernel on $\mathcal{M}$, as introduced in Section \ref{section:LapMfld}. With the choice \eqref{eq:definition:weights}, we call $\mathscr{L}_{\ell-1}^{\up}$ the empirical up Hodge Laplacian and we call $\langle\boldsymbol{\omega},\mathscr{L}_{\ell-1}^{\up}\boldsymbol{\omega}\rangle_n$ the empirical Dirchlet form.  Our analysis will be based on the following quantitative boundedness and smoothness conditions on the heat kernel $k_t$ and the functions $f_1,\dots,f_\ell$. 

\begin{assumption}[Global heat kernel bound]\label{Ass:GHB} There are constants $c_1,C_1>0$ such that 
    \begin{align*}
        k_t(x,y)\leq \frac{C_1}{t^{d/2}}\exp\Big(-c_1\frac{d_{\mathcal{M}}(x,y)^2}{t}\Big). 
    \end{align*}
for all $x,y\in\mathcal{M}$ and all $t\in(0,1]$. Here $d_{\mathcal{M}}$ denotes the intrinsic distance on $\mathcal{M}$.
\end{assumption}

\begin{assumption}[Smoothness properties]\label{Ass:SP} There is a constant $C_2>0$ such that
    \begin{align*}
        \|f_a\|_{L^\infty},\|\Delta (f_af_b)\|_{L^\infty},\Big\|\Big(\frac{e^{-t\Delta}-I+t\Delta}{t^2}\Big)(f_af_b)\Big\|_{L^\infty}\leq C_2,
    \end{align*}
    for all $0\leq a,b\leq \ell$, where $f_0\equiv 1$.
\end{assumption}

Note that both assumptions are satisfied for $\mathcal{M}$ closed (see \cite{MR1736868}) and smooth functions $f_1,\dots,f_\ell$. Their particular purpose is to introduce the constants $c_1, C_1, C_2$ that are important for our analysis. We now state our main result.

    \begin{thm}\label{thm:main:result}
    Let $t\in(0,1]$ and $A>0$ be real numbers and let $\ell\geq 0$ and $n\geq 2+2\ell$ be natural numbers. Moreover, suppose $\mathscr{L}_{\ell-1}^{\up}$ is the empirical up Hodge Laplacian based on the sample $X_1,\dots,X_n$ satisfying Assumption \ref{Ass:MH} and having the weights $(k_{i_0\cdots i_\ell})$ introduced in \eqref{eq:definition:weights}. Furthermore, let $f_1,\dots,f_\ell\in C^\infty(\mathcal{M})$ and let $\omega=f_1(df_2\wedge\dots\wedge df_\ell)$ and $\boldsymbol{\omega}=\mathbf{f_1}(\delta\mathbf{f}_2\wedge\dots\wedge \delta\mathbf{f_\ell})$ as introduced in \eqref{eq:definition:emprical:l:form}. Finally, let $\Delta^{\up}_{\ell-1}$ be the up Laplace-Beltrami operator on $\mathcal{M}$ and suppose that Assumptions \ref{Ass:GHB} and \ref{Ass:SP} are satisfied. Then, with probability at least $1-n^{-A}$, we have
    \begin{align*}
        &\Big|\langle\boldsymbol{\omega},\mathscr{L}_{\ell-1}^{\up}\boldsymbol{\omega}\rangle_n-\langle\omega,\Delta_{\ell-1}^{\up}\omega\rangle\Big|\leq C\Big(t+\sum_{j=1}^{\ell+1}\Big(\frac{(\log n)^{j/2}}{t^{d(j-1)/4}n^{j/2}}+\frac{(\log n)^{(j+1)/2}}{t^{d(j-1)/2}n^{(j+1)/2}}\Big)\Big),
    \end{align*}
    where $C>0$ is a constant depending only on $\ell,A,c_1,C_1,C_2$.
    \end{thm}

    The proof of Theorem \ref{thm:main:result} is given in Sections \ref{section:bias}--\ref{section:variance}. More precisely, in Section \ref{section:bias}, we study the approximation error $\langle\omega,\Delta_{\ell-1}^{\up}\omega\rangle-\mathbb{E} \langle\boldsymbol{\omega},\mathscr{L}_{\ell-1}^{\up}\boldsymbol{\omega}\rangle_n$, which relates Hodge theory on Riemannian manifolds to continuous Hodge theory on metric spaces. In Section \ref{section:variance}, we study the stochastic error $\langle\boldsymbol{\omega},\mathscr{L}_{\ell-1}^{\up}\boldsymbol{\omega}\rangle_n-\mathbb{E} \langle\boldsymbol{\omega},\mathscr{L}_{\ell-1}^{\up}\boldsymbol{\omega}\rangle_n$ using the machinery of U-statistics.

\subsection{Discussion}

\paragraph{Comparison to the literature} 
Let us compare our results with the literature, which has so far focused on the case $\ell =0$. For $\ell=0$, that is in the case of the (un-)normalized graph Laplacians, such and similar bounds have been studied previously in the context of Laplacian Eigenmaps and Diffusion Maps \cite{NIPS2006_5848ad95,Gine06,MR4130541,MR4491976,MR4452681}. Dirichlet error bounds provide a first important step towards the more sophisticated study of the spectral convergence of graph Laplacians towards the Laplace-Beltrami operator. A state-of-the-art bound in \cite{MR4452681} provides the Dirichlet error rate $\max(t,1/(nt^{d/2}))$ up to log-terms. In contrast, our theorem yield with high probability
    \begin{align*}
        \big|\langle f,\Delta f\rangle_n-\langle\mathbf{f},\mathscr{L}_{0}\mathbf{f}\rangle_n\big|\leq C\Big(t+\frac{\log^{1/2}n}{n^{1/2}}+\frac{\log n}{nt^{d/4}}+\frac{\log^{3/2} n}{n^{3/2}t^{d/2}}\Big).
    \end{align*}
    In particular, our result shows that the dimension dependence, that is the curse of dimensionality, only appears in lower order terms. An important question is to apply bias  reduction techniques to reduce the $Ct$ bias term. 

\paragraph{Directions for future research}
The present analysis provides a basis for further investigations of the problem of approximating the Laplace-Beltrami operator on $\ell$-forms by empirical Hodge Laplacians.

First, \Cref{thm:main:result} provides an concentration bound for the empirical Dirichlet form. This can be seen as a first step towards the study of eigenvalues and eigenforms. In the case of eigenvalues, this can be approached via the min-max characterizations. However, further steps must be implemented first. So far, \Cref{thm:main:result} is restricted to empirical $\ell$-forms in contrast to the set of all alternating functions and it depends on the strong smoothness properties in Assumption~\ref{Ass:SP}. 

Second, our results are stated in terms of the heat kernel, which is unknown to the statistician. It is possible to replace the heat kernel with a Gaussian kernel by combining the analysis in Section 3 of \cite{Wahl24} with Lemma \ref{lemma:representation:Dirichlet:form:empirical:l:form}, in order to obtain practical results. Since this requires several further assumptions on the local approximation of the intrinsic distance by the extrinsic distance and of the heat kernel by the geodesic kernel, we have not included this in the current paper and refer to subsequent work.


\section{The bias term: continuous Hodge theory}\label{section:bias}

\subsection{Main approximation bound}

In this section we relate $\langle\omega,\Delta_{\ell-1}^{\up}\omega\rangle$ to
\begin{align*}
    \mathbb{E} \langle\boldsymbol{\omega},\mathscr{L}_{\ell-1}^{\up}\boldsymbol{\omega}\rangle_n
    & =\mathbb{E}\frac{1}{\binom{n}{\ell+1}}\sum_{1\leq i_0<\dots<i_\ell\leq n}\frac{1}{\ell !(2t)^\ell}\Big(\frac{1}{\ell+1}\sum_{a=0}^\ell\prod_{\substack{b=0\\b\neq a}}^\ell k_t(X_{i_a},X_{i_b})\Big)\det_{\ell\times\ell}\big( \delta\mathbf{f}_a(X_{i_0},X_{i_b})\big)^2\\
    &=\frac{1}{\ell!(2t)^\ell}\int_{\mathcal{M}^{\ell+1}}\det_{\ell\times\ell}\big( \delta f_a(x,x_b)\big)^2\Big(\prod_{b=1}^\ell k_t(x,x_b)dx_b\Big)dx,
\end{align*}
where we applied Lemma \ref{lemma:representation:Dirichlet:form:empirical:l:form} and the choice of weights in \eqref{eq:definition:weights} in the first equality, and the symmetry of the involved squared determinant in the second equality. The main result of this section is the following error bound.

\begin{proposition}\label{prop:bias}
    Let $t\in(0,1]$, $f_1,\dots,f_\ell\in C^\infty(\mathcal{M})$, and $\omega=f_1\cdot(df_2\wedge\dots\wedge df_\ell)$. Suppose that Assumption \ref{Ass:SP} is satisfied. Then we have
    \begin{align*}
        \Big|\langle\omega,\Delta_{\ell-1}^{\up}\omega\rangle-\frac{1}{\ell!(2t)^\ell}\int_{\mathcal{M}^{\ell+1}}\limits\big(\det_{\ell\times\ell}( \delta f_a(x,x_b))\big)^2\Big(\prod_{b=1}^\ell k_t(x,x_b)dx_b\Big)dx\Big|\leq C t, 
    \end{align*}
    where $C>0$ is a constant depending only on $\ell$ and $C_2$. 
\end{proposition}

 To prove Proposition \ref{prop:bias}, it is necessary to relate differential Hodge theory to continuous Hodge theory \cite{Smale12,MR2876937,HK24,HinzKommer24}. See e.g.~Theorem 1 in \cite{Smale12} for the matching of cohomology and Corollary~5.2 in \cite{HinzKommer24} for a related pointwise non-local-to-local convergence result of cotangential structures.
 
\subsection{Technical lemmas}
The following lemma is a quantitative variant of the approximation of the carré du champ operator through the semigroup ( see Chapter 3 in \cite{MR3155209}). 
\begin{lemma}\label{lemma:tech:1}
    Under the assumptions of Proposition \ref{prop:bias}, we have
   \begin{align*}
            \Big|\langle df_a,df_b\rangle_x- \frac{1}{2t}\int_\mathcal{M}k_t(x,y)(f_a(x)-f_a(y))(f_b(x)-f_b(y))\,dy\Big|\leq Ct,
    \end{align*}
    for all $1\leq a,b\leq \ell$ and all $x\in\mathcal{M}$, where $C>0$ depends only on $C_2$.    
\end{lemma}

\begin{proof}
    By the product rule, we have
    \begin{align}\label{eq:chain:rule}
        \langle df_a,df_b\rangle_x&=\langle \nabla f_a,\nabla f_b\rangle_x=\frac{1}{2}\big(f_a\Delta f_b+f_b\Delta f_a-\Delta(f_af_b)\big)(x)=(*)+(**),
    \end{align}
    where 
    \begin{align*}
        (*)=&\frac{1}{2}\Big(-f_a\Big(\frac{e^{-t\Delta}-I}{t}\Big)f_b-f_b\Big(\frac{e^{-t\Delta}-I}{t}\Big)f_a+\Big(\frac{e^{-t\Delta}-I}{t}\Big)(f_af_b)\Big)(x),\\
        (**)=&\frac{1}{2}\Big(f_a\Big(\frac{e^{-t\Delta}-I+t\Delta}{t}\Big)f_b+f_b\Big(\frac{e^{-t\Delta}-I+t\Delta}{t}\Big)f_a\Big)(x)
        -\frac{1}{2}\Big(\Big(\frac{e^{-t\Delta}-I+t\Delta}{t}\Big)(f_af_b)\Big)(x),
    \end{align*}
    and where $\langle\cdot,\cdot\rangle_x$ also denotes the inner product on the tangent space at point $x$ (by some abuse of notation).
    Now, using \eqref{eq:integral:kernel} and the fact that $k_t(x,\cdot)$ integrates to $1$, we have
    \begin{align*}
        (*)=&-\frac{1}{2t}\int_{\mathcal{M}}k_t(x,y)f_a(x)f_b(y)\,dy-\frac{1}{2t}\int_{\mathcal{M}}k_t(x,y)f_a(y)f_b(x)\,dy\\
        &+\frac{1}{2t}\int_{\mathcal{M}}k_t(x,y)f_a(y)f_b(y)\,dy+\frac{1}{2t}\int_{\mathcal{M}}k_t(x,y)f_a(x)f_b(x)\,dy\\
        =&\frac{1}{2t}\int_\mathcal{M}k_t(x,y)(f_a(x)-f_a(y))(f_b(x)-f_b(y))\,dy.
    \end{align*}
    On the other hand, by Assumptions \ref{Ass:SP}, we have that $(**)$ is upper-bounded by $(C_2^2+C_2/2)t$ in absolute value.
\end{proof}

The following lemma is the key in order to connect the empirical Dirichlet form in Lemma \ref{lemma:representation:Dirichlet:form:empirical:l:form} to an integral approximation of $\langle\cdot,\cdot \rangle$. 

\begin{lemma}[Andréief's identity]\label{andreief:identity} Let $\nu$ be a measure on a measurable space $\mathcal{X}$, and let $\phi_1,\dots,\phi_\ell:\mathcal{X}\rightarrow \mathbb{R}$ be square-integrable. Then we have
\begin{align*}
    \det_{\ell\times \ell}\Big(\int_\mathcal{X}\phi_a(y)\phi_b(y)\,\nu(dy)\Big)=\frac{1}{\ell!}\int_{\mathcal{X}^\ell}\det_{\ell\times \ell}\left(\phi_a(y_b)\right)^2\,\nu(dy_1)\cdots \nu(dy_\ell).
\end{align*}
\end{lemma}
Andréief's identity is a continuous analog of the Cauchy-Binet formula. It is a standard technique in random matrix theory \cite{Con05}. 

\begin{lemma}\label{lemma:tech:2}
    Let $A,B\in\mathbb{R}^{\ell\times\ell}$ be matrices such that $|A_{ij}|\leq a$ and $|A_{ij}-B_{ij}|\leq Ct$ for all $1\leq i,j\leq \ell$ and real numbers $a>0$, $C\geq 1$, and $t\in(0,1]$. Then we have 
    \begin{align*}
        \big|\det(B)-\det(A)\big|\leq (C(a+1))^\ell\ell!t.
    \end{align*}
\end{lemma}

Lemma \ref{lemma:tech:2} follows from an elementary computation using the Leibniz formula for determinants and the binomial formula, and is omitted. We now turn to the proof of Proposition \ref{prop:bias}.

\begin{proof}[Proof of Proposition \ref{prop:bias}]
    By the definition of $\Delta_{\ell-1}^{\up}$ and the properties of the inner product $\langle\cdot,\cdot\rangle$ on $\Omega^{\ell-1}(\mathcal{M})$ defined in \Cref{section:LapMfld}, we have
    \begin{align*}
        \langle\omega,\Delta_{\ell-1}^{\up}\omega\rangle&=\langle d_{\ell-1}\omega,d_{\ell-1}\omega\rangle=\langle df_1\wedge\dots\wedge df_\ell,df_1\wedge\dots\wedge df_\ell\rangle=\int_{\mathcal{M}}\det_{\ell\times\ell}\big(\langle df_a,df_b\rangle_x\big)dx
    \end{align*}
    By Lemma \ref{lemma:tech:1}, \eqref{eq:chain:rule}, and Assumption \ref{Ass:SP}, we have 
    \begin{align*}
            &\Big|\langle df_a,df_b\rangle_x- \frac{1}{2t}\int_\mathcal{M}k_t(x,y)\delta f_a(x,y)\delta f_b(x,y)\,dy\Big|\leq Ct, \ \text{ and }\ \big|\langle df_a,df_b\rangle_x\big|\leq C,
    \end{align*}
    for all $1\leq a,b\leq \ell$ and some constant $C>0$ depending only on $C_2$. Hence, Lemma \ref{lemma:tech:2}
 yields
 \begin{align}\label{Proof:approximation:error:1}
     \Big|\det_{\ell\times \ell}\big(\langle df_a,df_b\rangle_x\big)-\det_{\ell\times \ell}\Big(\frac{1}{2t}\int_\mathcal{M}\delta f_a(x,y)\delta f_b(x,y)k_t(x,y) dy\Big)\Big|\leq Ct,
 \end{align}
 where $C>0$ depends only on $C_2$. By Andréief's identity 
 \begin{align}\label{Proof:approximation:error:2}
     \det_{\ell\times \ell}\Big(\frac{1}{2t}\int_\mathcal{M}\delta f_a(x,y)\delta f_b(x,y)k_t(x,y)dy\Big)=\frac{1}{\ell !}\frac{1}{(2t)^\ell}\int_{\mathcal{M}^\ell}\det_{\ell\times\ell}\big( \delta_0f_a(x,x_b)\big)^2\Big(\prod_{b=1}^\ell k_t(x,x_b)dx_b\Big).
 \end{align}
 Inserting \eqref{Proof:approximation:error:2} into \eqref{Proof:approximation:error:1}, the claim follows from integrating with respect to $x$ and the triangle inequality.
 \end{proof}

\section{The variance term: concentration of U-statistics}\label{section:variance}

\subsection{Main concentration bound}

In this section we study the stochastic error $\langle\boldsymbol{\omega},\mathscr{L}_{\ell-1}^{\up}\boldsymbol{\omega}\rangle_n-\mathbb{E} \langle\boldsymbol{\omega},\mathscr{L}_{\ell-1}^{\up}\boldsymbol{\omega}\rangle_n$ using the theory of U-statistics \cite{Lee90,PG99}. More precisely, we analyze the expression 
\begin{align}\label{eq:U:statistics}
    U_{n}(\ell,t)=\frac{1}{\binom{n}{\ell+1}}\sum_{1\leq i_0<\dots<i_\ell\leq n}h_t(X_{i_0},\dots,X_{i_\ell})
\end{align}
with 
\begin{align*}
    &h_t(x_0,\dots,x_\ell)=\Big(\frac{1}{\ell+1}\sum_{a=0}^\ell\prod_{\substack{b=0\\b\neq a}}^\ell\frac{1}{t}k_t(x_a,x_b)\Big)\cdot (D(f_1,\dots,f_\ell,x_0,\dots,x_\ell))^2
\end{align*}
and
\begin{align*}
    D(f_1,\dots,f_\ell,x_0,\dots,x_\ell)=\det_{\ell\times\ell}(\delta f_a(x_0,x_b)).
\end{align*}
The main result of this section is the following concentration bound.

\begin{proposition}\label{Prop:conc:U:statistic}
    Let $t\in(0,1]$, $A>1$, and $f_1,\dots,f_\ell\in C^\infty(\mathcal{M})$. Suppose that $n\geq 2+2\ell$ and that Assumptions \ref{Ass:MH}--\ref{Ass:SP} are satisfied. Then, with probability at least $1-n^{-A}$, we have
    \begin{align*}
        \Big| U_{n}(\ell,t)&-\frac{1}{t^\ell}\int_{\mathcal{M}^{\ell+1}}(D(f_1,\dots,f_\ell,x_0,\dots,x_\ell))^2\Big(\prod_{b=1}^\ell k_t(x_0,x_b)dx_b\Big)dx_0 \Big|\\
        &\leq C\sum_{j=1}^{\ell+1}\Big(\frac{(\log n)^{j/2}}{t^{d(j-1)/4}n^{j/2}}+\frac{(\log n)^{(j+1)/2}}{t^{d(j-1)/2}n^{(j+1)/2}}\Big),
    \end{align*}
    where $C$ is a constant depending only on $\ell, A, c_1,C_1$ and $C_2$.
\end{proposition}

\subsection{The Hoeffding decomposition}

First, note that $h_t$ is symmetric, so that expression \eqref{eq:U:statistics} is an U-statistic of order $\ell+1$. Since $h_t$ is not degenerate, we have to apply the Hoeffding decomposition before we can proceed with the proof of Proposition \ref{Prop:conc:U:statistic}. 

We start with some notation. For a function $f$ on $\mathcal{M}$ we write $Pf=\mathbb{E}f(X)=\int_\mathcal{M}f(x)\,dx$. For a symmetric function $h:\mathcal{M}^{\ell+1}\rightarrow \mathbb{R}$ and $0\leq j\leq \ell$, we write 
\begin{align*}
    P^{\ell-j}h(x_0,\dots,x_j)&=\mathbb{E}f(x_0,\dots,x_j,X_{j+1},\dots,X_\ell)=\int_{\mathcal{M}^{\ell-j}}f(x_0,\dots,x_\ell)\,dx_{j+1}\dots dx_\ell.
\end{align*}
We set
\begin{align*}
    h_t^{(j)}=(\delta_{x_0}-P)\times \dots\times (\delta_{x_j}-P) \times P^{\ell-j} h_t,\qquad x_0,\dots,x_j\in\mathcal{M}.
\end{align*}
Then $h_t^{(j)}$ is a symmetric and degenerate kernel, that is
\begin{align*}
    \mathbb{E} h_t^{(j)}(x_0,\dots,x_{j-1},X_j)=\int_\mathcal{M}h_t^{(j)}(x_0,\dots,x_{j-1},x_j)\,dx_j=0
\end{align*}
for all $x_0,\dots,x_{j-1}\in\mathcal{M}$ and the Hoeffding decomposition, see \cite[Section 3.5]{PG99} or \cite{Lee90}, implies that 
\begin{align}\label{eq:Hoeffding:decomposition}
    &\frac{1}{\binom{n}{\ell+1}}\sum_{1\leq i_0<\dots<i_\ell\leq n}h_t(X_{i_0},\dots,X_{i_\ell})-\int_{\mathcal{M}^{\ell+1}}h_t(x_0,\dots,x_\ell)\,dx_{0}\dots dx_\ell\\
    &=\sum_{j=0}^\ell \frac{\binom{\ell+1}{j+1}}{\binom{n}{j+1}}\sum_{1\leq i_0<\dots<i_j\leq n}h_t^{(j)}(X_{i_0},\dots,X_{i_j}). \nonumber
\end{align}
Here, the expressions 
\begin{align}
    \frac{1}{\binom{n}{j+1}}\sum_{1\leq i_0<\dots<i_j\leq n}h_t^{(j)}(X_{i_0},\dots,X_{i_j})
\end{align}
are degenerate U-statistics, to which we can apply the machinery of U-statistics \cite{PG99,Mins24}, provided that we have upper bounds for $\|h_t^{(j)}\|_{L^\infty}$ and $\|h_t^{(j)}\|_{L^2}$.

\subsection{Bounding the degenerate kernel}

The following preliminary lemma combines smoothness properties of the functions with heat kernel estimates.

\begin{lemma}\label{lemma:bound:kernel:smoothness}
    Under the assumptions of Proposition \ref{Prop:conc:U:statistic}, we have 
    \begin{align*}
        \frac{1}{t}k_t(x,y)(f_b(x)-f_b(y))^2\leq C\frac{1}{t^{d/2}}
    \end{align*}
    and
    \begin{align*}
        \int_{\mathcal{M}}\frac{1}{t}k_t(x,y)(f_b(x)-f_b(y))^2\,dy\leq C
    \end{align*}
    for all $x,y\in\mathcal{M}$, all $b=1,\dots,\ell$, and all $t\in(0,1]$, where $C$ is a constant depending only on $c_1,C_1,C_2$.
\end{lemma}

\begin{proof}
    First, by Assumption \ref{Ass:SP} and \eqref{eq:chain:rule}, we have $\|\nabla f_b\|_x^2=\langle \nabla f_b,\nabla f_b\rangle_x\leq C_2^2+C_2/2$ for all $x\in\mathcal{M}$. From this it follows that $f_b$ is a Lipschitz function on $(\mathcal{M},d_{\mathcal{M}})$ with Lipschitz constant bounded by $C_2^2+C_2/2$.
    From this and Assumptions \ref{Ass:GHB}, we get
    \begin{align*}
        \frac{1}{t}k_t(x,y)(f_b(x)-f_b(y))^2&\leq \frac{C_1(C_2^2+C_2/2)}{t^{d/2}}\exp\Big(-c_1\frac{d_{\mathcal{M}}^2(x,y)}{t}\Big)\frac{d_{\mathcal{M}}^2(x,y)}{t}\leq \frac{C_1(C_2^2+C_2/2)}{ec_1}\frac{1}{t^{d/2}},
    \end{align*}
    where we used the inequality $xe^{-x}\leq e^{-1}$, $x\geq 0$. Second, by Assumption \ref{Ass:SP}, we have
    \begin{align*}
        &\int_{\mathcal{M}}\frac{1}{t}k_t(x,y)(f_b(x)-f_b(y))^2\,dy=\Big(\Big(\frac{e^{-t\Delta}-I}{t}\Big)f_b^2-2f_b\Big(\frac{e^{-t\Delta}-I}{t}\Big)f_b\Big)(x)\leq C_2+2C_2^2.
    \end{align*} 
    This completes the proof.
\end{proof}

\begin{lemma}\label{lemma:bounds:apply:U:concentation}
    Suppose that the assumptions of Proposition \ref{Prop:conc:U:statistic} are satisfied. Let $J\subseteq \{0,\dots,\ell\}$ be a nonempty subset. Then we have
    \begin{align}\label{eq:bounds:apply:U:concentation:1}
    \int_{\mathcal{M}^{|J^\complement|}}\Big(\frac{1}{t^\ell}\prod_{b=1}^\ell k_t(x_0,x_b)\Big) D^2(f_1,\dots,f_\ell,x_0,\dots,x_\ell)\, dx_{J^\complement}\leq C\frac{1}{t^{d(|J|-1)/2}}    
    \end{align}
    for all $(x_b)_{b\in J}$ and
    \begin{align}\label{eq:bounds:apply:U:concentation:2}
    &\Big(\int_{\mathcal{M}^{|J|}}\Big(\int_{\mathcal{M}^{|J^\complement|}}\Big(\frac{1}{t^\ell}\prod_{b=1}^\ell k_t(x_0,x_b)\Big) D^2(f_1,\dots,f_\ell,x_0,\dots,x_\ell)\, dx_{J^\complement}\Big)^2dx_J\Big)^{1/2}\leq C\frac{1}{t^{d(|J|-1)/4}},
    \end{align}
    where $C>0$ is a constant depending only on $c_1,C_1,C_2$ and $\ell$. Here, $dx_{J^\complement}$ means $\prod_{b\in J^\complement}dx_b$ and $dx_J$ means $\prod_{b\in J}dx_b$.
\end{lemma}

\begin{proof}
    Using the Leibniz formula applied to the transpose, we have
    \begin{align}\label{eq:proof:bounds:apply:U:concentation:1}
        \begin{split}
        D^2(f_1,\dots,f_\ell,x_0,\dots,x_\ell) &=\sum_{\sigma,\tau\in S_\ell}\prod_{b=1}^\ell (f_{\sigma(b)}(x_b)-f_{\sigma(b)}(x_0))(f_{\tau(b)}(x_b)-f_{\tau(b)}(x_0))\\
        &\leq \ell! \sum_{\sigma\in S_\ell}\prod_{b=1}^\ell (f_{\sigma(b)}(x_b)-f_{\sigma(b)}(x_0))^2,
        \end{split}
    \end{align}
    where we used the inequality $xy\leq (x^2+y^2)/2$ and the symmetry in $\sigma,\tau\in S_\ell$. We now consider separately the two cases $0\in J$ and $0\notin J$. 
    
    First, let $0\in J$. Inserting \eqref{eq:proof:bounds:apply:U:concentation:1} into \eqref{eq:bounds:apply:U:concentation:1} and using the Fubini theorem, we get
    \begin{align}
        &\int_{\mathcal{M}^{|J^\complement|}}\Big(\frac{1}{t^\ell}\prod_{b=1}^\ell k_t(x_0,x_b)\Big)\cdot D^2(f_1,\dots,f_\ell,x_0,\dots,x_\ell)\, dx_{J^\complement}\label{eq:proof:bounds:apply:U:concentation:2}\\
        &\leq \ell !\sum_{\sigma\in S_\ell}\int_{\mathcal{M}^{|J^\complement|}}\prod_{b=1}^\ell\frac{1}{t} k_t(x_0,x_b) (f_{\sigma(b)}(x_b)-f_{\sigma(b)}(x_0))^2 dx_{J^\complement}\nonumber\\
        &=\ell !\sum_{\sigma\in S_\ell} \prod_{\substack{b\in J\\ b\neq 0}}\frac{1}{t} k_t(x_0,x_b) (f_{\sigma(b)}(x_b)-f_{\sigma(b)}(x_0))^2\cdot\prod_{b\in J^\complement}\int_{\mathcal{M}}\frac{1}{t} k_t(x_0,x_b) (f_{\sigma(b)}(x_b)-f_{\sigma(b)}(x_0))^2dx_b\Big).\nonumber
    \end{align}
    Inserting Lemma \ref{lemma:bound:kernel:smoothness} the first claim follows in this case. Similarly, 
    \begin{align*}
        &\int_{\mathcal{M}^{|J|}}\Big(\int_{\mathcal{M}^{|J^\complement|}}\Big(\frac{1}{t^\ell}\prod_{b=1}^\ell k_t(x_0,x_b)\Big)\cdot D^2(f_1,\dots,f_\ell,x_0,\dots,x_\ell)\, dx_{J^\complement}\Big)^2dx_J\\
        & \leq (\ell!)^3 \sum_{\sigma\in S_\ell}\int_{\mathcal{M}^{|J|}}\Big(\int_{\mathcal{M}^{|J^\complement|}} \prod_{b=1}^\ell\frac{1}{t} k_t(x_0,x_b) (f_{\sigma(b)}(x_b)-f_{\sigma(b)}(x_0))^2\, dx_{J^\complement}\Big)^2dx_J.
    \end{align*}
    Proceeding as in \eqref{eq:proof:bounds:apply:U:concentation:2}, applying Lemma \ref{lemma:bound:kernel:smoothness} and using the fact that we can also integrate with respect to $x_J$,  the second claim follows. 
    
    Second, let $0\in J^\complement$. Moreover, let $a\in J$ be arbitrary but fixed. Then 
    \begin{align*}
        \int_{\mathcal{M}^{|J^\complement|}}
         \Big( \frac{1}{t^\ell}\prod_{b=1}^\ell k_t(x_0,x_b)\Big)&
        \cdot D^2(f_1,\dots,f_\ell,x_0,\dots,x_\ell)\, dx_{J^\complement}
        \\\leq \ell !\sum_{\sigma\in S_\ell}\int_{\mathcal{M}}\Big[&\prod_{\substack{b\in J\setminus\{a\}}}\frac{1}{t} k_t(x_0,x_b) (f_{\sigma(b)}(x_b)-f_{\sigma(b)}(x_0))^2  
        \\\cdot&\prod_{b\in J^c\setminus\{0\}}\int_{\mathcal{M}}\frac{1}{t} k_t(x_0,x_b) (f_{\sigma(b)}(x_b)-f_{\sigma(b)}(x_0))^2dx_b\\
          \cdot & \frac{1}{t} k_t(x_0,x_a) (f_{\sigma(a)}(x_a)-f_{\sigma(a)}(x_0))^2\Big]\,dx_0
    \end{align*}
    Inserting Lemma \ref{lemma:bound:kernel:smoothness}, the first claim follows in this case, ensuring that we integrate with respect to the $x_0$ in the last step. Similarly,
    \begin{align*}
        &\int_{\mathcal{M}^{|J|}}  \Big[ \int_{\mathcal{M}^{|J^\complement|}} \Big( \frac{1}{t^\ell}\prod_{b=1}^\ell k_t(x_0,x_b) \Big) \cdot D^2(f_1,\dots,f_\ell,x_0,\dots,x_\ell)\, dx_{J^\complement} \Big]^2dx_J\\
        & \leq (\ell!)^3 \sum_{\sigma\in S_\ell}\int_{\mathcal{M}^{|J|}} \Big[\int_{\mathcal{M}^{|J^\complement|}} \Big( \prod_{b=1}^\ell\frac{1}{t} k_t(x_0,x_b) (f_{\sigma(b)}(x_b)-f_{\sigma(b)}(x_0))^2\Big) \, dx_{J^\complement}\Big]^2dx_J\\
        & \leq C(\ell!)^3 \sum_{\sigma\in S_\ell} \int_{\mathcal{M}^{|J|+1}}
        \Big[\int_{\mathcal{M}^{|J^\complement|-1}} \prod_{\substack{b=1\\b\neq a}}^\ell\frac{1}{t} k_t(x_0,x_b) (f_{\sigma(b)}(x_b)-f_{\sigma(b)}(x_0))^2\, dx_{J^\complement\setminus\{0\}}\Big]^2\\
        &\qquad \qquad\qquad\qquad\quad   \cdot\frac{1}{t} k_t(x_0,x_a) (f_{\sigma(a)}(x_a)-f_{\sigma(a)}(x_0))^2dx_0dx_J,
    \end{align*} 
where we applied the Cauchy-Schwarz inequality and Lemma \ref{lemma:bound:kernel:smoothness} in the last inequality. Now, we can proceed as in the first case.
\end{proof}

\begin{corollary}\label{corollary:bound:kernel:HD}
    Suppose that the assumptions of Proposition \ref{Prop:conc:U:statistic} are satisfied. Then, for each $j=0,\dots,\ell$, 
    \begin{align*}
        \|h_t^{(j)}\|_{L^\infty}&\leq C\frac{1}{t^{dj/2}},\\
        \|h_t^{(j)}\|_{L^2}&\leq C\frac{1}{t^{dj/4}}.
    \end{align*}
\end{corollary}

\begin{proof}
    By construction, we have 
    \begin{align*}
        h_t^{(j)}&=(\delta_{x_0}-P)\times \dots\times (\delta_{x_j}-P) \times P^{\ell-j} h_t=\sum_{J\subseteq \{0,\dots,j\}}(-1)^{j+1-|J|}\prod_{b\in J}\delta_{x_b}\times P^{\ell+1-|J|}h_t\\
        &=\sum_{J\subseteq \{0,\dots,j\}}\frac{(-1)^{j+1-|J|}}{\ell+1}\sum_{a=0}^\ell \int_{\mathcal{M}^{|J^\complement|}}\Big(\frac{1}{t^\ell}\prod_{\substack{b=0\\ b\neq a}}^\ell k_t(x_a,x_b)\Big) D^2(f_1,\dots,f_\ell,x_0,\dots,x_\ell)dx_{J^\complement}.
    \end{align*}
    The first claim follows from Lemma \ref{lemma:bounds:apply:U:concentation}, taking into account that $J=\{0,\dots,j\}$ provides the bound $Ct^{-dj/2}$ with the highest exponent and thus the dominating part because $t\in(0,1]$, and each summand with $a>0$ can be reduced to $a=0$ by relabeling the variables and using the alternating property of $D$. The second claim follows similarly from Minkowski's inequality and the second claim in Lemma \ref{lemma:bounds:apply:U:concentation}.    
\end{proof}

\begin{proof}[Proof of Proposition \ref{Prop:conc:U:statistic}]
    Using \eqref{eq:Hoeffding:decomposition} and Corollary \ref{corollary:bound:kernel:HD}, the concentration behavior of \eqref{eq:U:statistics} can be analyzed using the standard machinery for U-statistics. We follow the strategy of \cite{Mins24}. Let $\epsilon_1,\dots,\epsilon_n$ be independent Rademacher random variables independent of $X_1,\dots,X_n$. Then, by symmetrization (see \cite{SCK19} or \cite[Theorem 3.1]{PG99} for a result with slightly worse constants) and the Bonami inequality (\cite[Theorem 3.22]{PG99}), we have for $j=0,\dots,\ell$, 
    \begin{align*}
        &\mathbb{E}^{1/p}\Big|\frac{1}{\binom{n}{j+1}^{1/2}}\sum_{1\leq i_0<\dots<i_j\leq n}h_t^{(j)}(X_{i_0},\dots,X_{i_j})\Big|^p\\
        &\leq 2^{j+1}\mathbb{E}^{1/p}\Big|\frac{1}{\binom{n}{j+1}^{1/2}}\sum_{1\leq i_0<\dots<i_j\leq n}\epsilon_{i_0}\cdots\epsilon_{i_j}h_t^{(j)}(X_{i_0},\dots,X_{i_j})\Big|^p\\
        &\leq 2^{j+1}(p-1)^{\frac{j+1}{2}}\mathbb{E}^{1/p}\Big|\frac{1}{\binom{n}{j+1}}\sum_{1\leq i_0<\dots<i_j\leq n}(h_t^{(j)}(X_{i_0},\dots,X_{i_j}))^2\Big|^{p/2}.
    \end{align*}
    Next, we use a decoupling trick. For this, let $m$ be the largest integer such that $(j+1)m\leq n$. Then 
    \begin{align*}
        \frac{1}{\binom{n}{j+1}}\sum_{1\leq i_0<\dots<i_j\leq n}h_t^{(j)}(X_{i_0},\dots,X_{i_j})^2=\frac{1}{n!}\sum_{\sigma\in S_n}\frac{1}{m}\Big(\sum_{k=1}^m h_t^{(j)}(X_{\sigma((k-1)(j+1)+1)},\dots,X_{\sigma(k(j+1))})^2\Big)
    \end{align*}
    and thus by Jensen's inequality 
    \begin{align*}
        \mathbb{E}^{1/p}\Big(\frac{1}{\binom{n}{j+1}}\sum_{1\leq i_0<\dots<i_j\leq n}h_t^{(j)}(X_{i_0},\dots,X_{i_j})^2\Big)^{p/2}\leq \frac{1}{\sqrt{m}}\mathbb{E}^{1/p}\Big(\sum_{k=1}^m h_t^{(j)}(X_{(k-1)(j+1)+1},\dots,X_{k(j+1)})^2\Big)^{p/2}.
    \end{align*}
    Applying a moment inequality for nonnegative random variables \cite[Theorem 15.10]{BLM13}, we get
    \begin{align*}
        &\frac{1}{\sqrt{m}}\mathbb{E}^{1/p}\Big(\sum_{k=1}^mh_t^{(j)}(X_{(k-1)(j+1)+1},\dots,X_{k(j+1)})^2\Big)^{p/2}\\
        & \leq \frac{1}{\sqrt{m}}\Big(2m\mathbb{E}h_t^{(j)}(X_{1},\dots,X_{j+1})^2\Big)^{1/2}+\frac{1}{\sqrt{m}}\Big(\frac{p\sqrt{e}}{2}\mathbb{E}^{2/p}\max_{1\leq k\leq m}h_t^{(j)}(X_{(k-1)(j+1)+1},\dots,X_{k(j+1)})^p\Big)^{1/2}\\
        &\leq \sqrt{2}\|h^{(j)}\|_{L^2}+\frac{\sqrt{p}}{\sqrt{m}}\|h^{(j)}\|_{L^\infty}.
    \end{align*}
    Combining the above with Corollary \ref{corollary:bound:kernel:HD}, we arrive at
    \begin{align*}
        \mathbb{E}^{1/p}\Big|\frac{1}{\binom{n}{j+1}^{1/2}}\sum_{1\leq i_0<\dots<i_j\leq n}h_t^{(j)}(X_{i_0},\dots,X_{i_j})\Big|^p \leq 2^{j+1}(p-1)^{\frac{j+1}{2}}\Big(\sqrt{2}C\frac{1}{t^{dj/4}}+C\frac{\sqrt{p}}{\sqrt{m}}\frac{1}{t^{dj/2}}\Big).
    \end{align*}
    Now, since $n-j-1\geq n/2$ by assumption, we have $m(j+1)\geq n-j-1\geq n/2$, that is $m\geq n/(2(j+1))$, as well as
    \begin{align*}
        \binom{n}{j+1}=\frac{n\cdots (n-j)}{(j+1)!}\geq \Big(\frac{n}{2}\Big)^{j+1}\frac{1}{(j+1)!}.
    \end{align*}
    We conclude that 
    \begin{align*}
        \mathbb{E}^{1/p}\Big|\frac{1}{\binom{n}{j+1}^{1/2}}\sum_{1\leq i_0<\dots<i_j\leq n}h_t^{(j)}(X_{i_0},\dots,X_{i_j})\Big|^p
        &\leq C\Big(\frac{p^{\frac{j+1}{2}}}{t^{\frac{dj}{4}}n^{j+1}}+\frac{p^{\frac{j+2}{2}}}{t^{\frac{dj}{2}}n^{\frac{j+2}{2}}}\Big).
    \end{align*}
    Inserting these bounds into \eqref{eq:Hoeffding:decomposition} and using Minkowski's inequality, we get
    \begin{align*}
        \mathbb{E}^{1/p}\Big|U_{n}(\ell,t)&-\int_{\mathcal{M}^{\ell+1}}\Big(\det_{\ell\times\ell}(f_a(x_b)-f_a(x_0))\Big)^2\Big(\frac{1}{t^\ell}\prod_{b=1}^\ell k_t(x_0,x_b)dx_b\Big)dx_0\Big|^p\\
        &\leq C\sum_{j=0}^{\ell}\Big(\frac{p^{\frac{j+1}{2}}}{t^{\frac{dj}{4}}n^{j+1}}+\frac{p^{\frac{j+2}{2}}}{t^{\frac{dj}{2}}n^{\frac{j+2}{2}}}\Big).
    \end{align*}
    Inserting this into Markov's inequality
    \begin{align*}
        &\mathbb{P}\Big(\Big|U_{n}(\ell,t)-\int_{\mathcal{M}^{\ell+1}}\Big(\det_{\ell\times\ell}(f_a(x_b)-f_a(x_0))\Big)^2\Big(\frac{1}{t^\ell}\prod_{b=1}^\ell k_t(x_0,x_b)dx_b\Big)dx_0\Big|>u\Big)\\
        &\leq \frac{1}{u^p}\mathbb{E}\Big|U_{n}(\ell,t)-\int_{\mathcal{M}^{\ell+1}}\Big(\det_{\ell\times\ell}(f_a(x_b)-f_a(x_0))\Big)^2\Big(\frac{1}{t^\ell}\prod_{b=1}^\ell k_t(x_0,x_b)dx_b\Big)dx_0\Big|^p
    \end{align*}
    the claim follows from the choices $p=\log n$ and 
    \begin{align*}
        u=e^{A}C\sum_{j=0}^{\ell}\Big(\frac{p^{\frac{j+1}{2}}}{t^{\frac{dj}{4}}n^{j+1}}+\frac{p^{\frac{j+2}{2}}}{t^{\frac{dj}{2}}n^{\frac{j+2}{2}}}\Big).
    \end{align*}
\end{proof}

\begin{proof}[End of the proof of Theorem \ref{thm:main:result}]
    Decomposing $\langle\omega,\Delta_{\ell-1}^{\up}\omega\rangle-\langle\boldsymbol{\omega},\mathscr{L}_{\ell-1}^{\up}\boldsymbol{\omega}\rangle_n$ into the bias term $\langle\omega,\Delta_{\ell-1}^{\up}\omega\rangle-\mathbb{E} \langle\boldsymbol{\omega},\mathscr{L}_{\ell-1}^{\up}\boldsymbol{\omega}\rangle_n$ and the variance term $\mathbb{E} \langle\boldsymbol{\omega},\mathscr{L}_{\ell-1}^{\up}\boldsymbol{\omega}\rangle_n-\langle\boldsymbol{\omega},\mathscr{L}_{\ell-1}^{\up}\boldsymbol{\omega}\rangle_n$, \Cref{thm:main:result} follows from inserting \Cref{prop:bias,Prop:conc:U:statistic} and the triangle inequality.
\end{proof}

\section{Experiments}
\label{section:experiments}

We compare the spectra of the empirical operators $\mathscr{L}_1^{\up}$ and $\mathscr{L}_1^{\down}$ with the corresponding population eigenvalues of $\Delta_1^{\up}$ and $\Delta_1^{\down}$ on the unit volume sphere $r\cdot S^2$ with radius $r = (4\pi)^{-1/2}.$\\ 

\noindent\textbf{Population spectrum and heat kernel.} The spectrum of $\Delta_{1}$ for the unit sphere $S^2$ is obtained in \cite[Theorems A-C]{Folland89} (see additionally \cite[Chapter 4]{MR304972} for the multiplicities). For each fixed $r > 0,$ the nonzero spectra of both $\Delta_1^{\up}$ and $\Delta_1^{\down}$ of $r\cdot S^2$ are the same. In each case, for $j \geq 1,$ the $j$-th positive eigenvalue is $j(j+1)/r^2$ and has multiplicity $2j+1$ and the spectrum of $\Delta_1$ is the multiset union of the non-zero spectra of $\Delta_1^{\up}$ and $\Delta_1^{\down}.$ Table \ref{fig.table.eigenvalues} below lists the corresponding first $4$ eigenvalues of $\Delta_1.$ We will denote by $\lambda_j^{\up}$ and $\lambda_j^{\down}$ the ascending positive eigenvalues, counted with multiplicities, of $\Delta_1^{\up}$ and $\Delta_1^{\down},$ respectively.

The heat kernel of $r\cdot S^2$ admits a Mercer expansion in terms of spherical harmonics \cite{zhao.fams.2018.00001, atkinson2012spherical}. Rescaling and applying the addition theorem for spherical harmonics \cite{Male_ek_2001}, leads to 
\begin{align}
    \label{eq.heat.kernel.sphere}
    k_t(x, y) = \sum_{j = 0}^\infty \frac{2j+1}{4\pi r^2}e^{-\tfrac{j(j+1)t}{r^2}}P_j\Big(\frac{\langle x, y\rangle}{r^2}\Big), \qquad x, y \in r\cdot S^2,
\end{align}
where $P_j$ is the $j$-th Legendre polynomial.

\begin{figure}[H]
\centering

\begin{minipage}{0.6\textwidth}
\centering
\begin{tabular}{|c|c|ccc|}
\hline
& & \multicolumn{3}{c|}{Multiplicities} \\
Index & Eigenvalue & $\Delta_1^{\down}$ & $\Delta_1^{\up}$ & $\Delta_1$ \\
\hline
1 & $2/r^2$ & $3$ & $3$ & $6$\\
2 & $6/r^2$ & $5$ & $5$ & $10$ \\
3 & $12/r^2$ & $7$ & $7$ & $14$\\
4 & $20/r^2$ & $9$ & $9$ & $18$\\
\hline
\end{tabular}
\captionof{table}{First eigenvalues of $\Delta_1$ on $r\cdot S^2,$ broken \\ down by contributions from $\Delta_{1}^{\up}$ and $\Delta_1^{\down}.$}
\label{fig.table.eigenvalues}
\end{minipage}
\hfill
\begin{minipage}{0.39\textwidth}
\centering
\includegraphics[height=6.23\baselineskip, trim=2cm 8cm 2cm 8cm, clip]{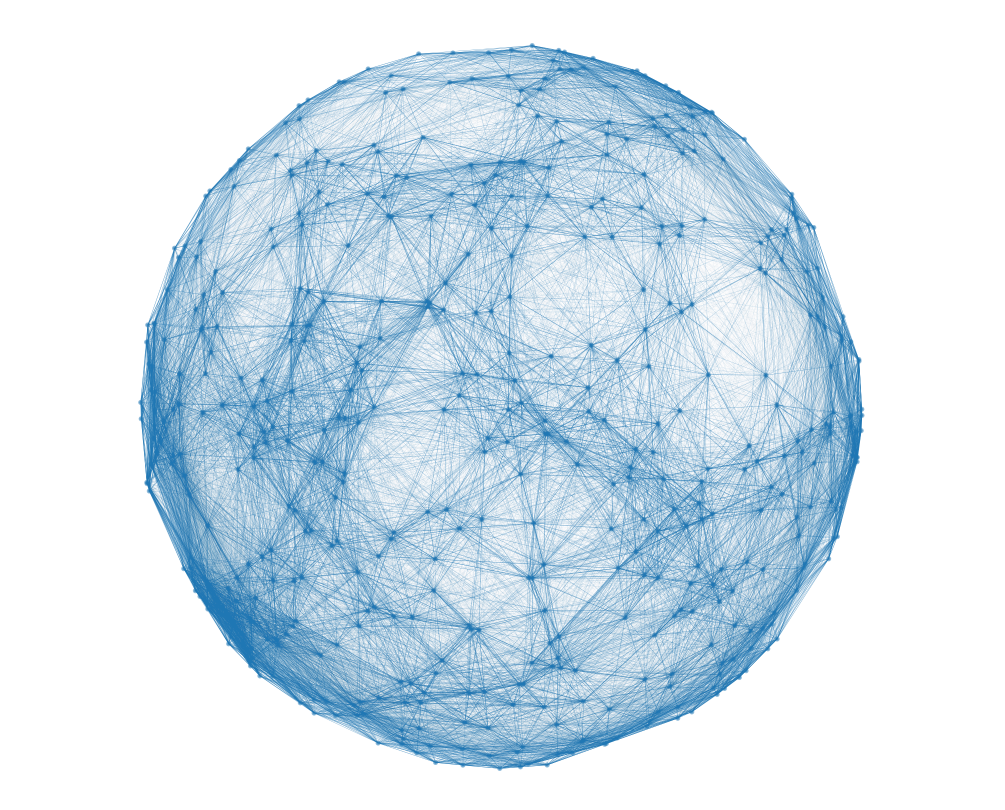}
\caption{Weighted graph of $V$. Edge thickness is proportional to $k_{ij}^{(\alpha)}$.}
\label{fig.data.cloud}
\end{minipage}

\end{figure}

\noindent\textbf{Eigenvalue problem.} We sample $X_1, \dots, X_n$ as i.i.d.\ random variables uniformly distributed on $r\cdot S^2.$ The spaces $L^2_{\wedge}(V^{\ell + 1})$ are identified with $\mathbb{R}^{d_\ell}$ of dimension $d_\ell:={ n\choose \ell + 1}$ with the canonical basis indexed by ordered $(\ell + 1)$-tuples of distinct elements of $\{1, \dots, n\}.$ Let $B_1 \in \mathbb{R}^{d_{\ell +1} \times d_\ell}$ denote the matrix representation of $\delta_1$ in these bases, where the $(t,e)$-entry of $B_1$ is zero if the edge corresponding to $e$ is not included in the triangle corresponding to $\tau$ and equals to the sign induced by \eqref{eq:Helm:up} otherwise (see also \cite{Lim20, doi:10.1137/22M1482299, MR4164275} for more details). The weighted inner products on $L^2_{\wedge}(V^2)$ and $L^2_{\wedge}(V^3)$ are represented by diagonal matrices $W_1$ and $W_2,$ respectively. Since the operator $\delta_1^*$ is the adjoint of $\delta_1,$ it has the matrix representation $W_1^{-1}B_1^\top W_2$. Consequently, $\mathscr{L}_1^{\up}=\delta_1^*\delta_1$ has the matrix representation $W_1^{-1}B_1^\top W_2B_1$. Taking the weighted inner product into account, we aim at obtaining the eigenvalues of $W_1^{-1/2}B_1^\top W_2B_1W_1^{-1/2},$ which can be obtained by solving the generalized eigenvalue problem
\begin{align}
    \label{eq.gen.ev.problem}
    B_1^TW_2B_1w = \lambda W_1w, \qquad (\lambda, w) \in \mathbb{R}\times \mathbb{R}^{d_{\ell}}.
\end{align}
 To construct $W_1$ and $W_2$, we approximate the true heat kernel by the $50$ leading terms of \eqref{eq.heat.kernel.sphere}, and the weights $(k_{i_0\dots i_{\ell}}), \ \ell \geq 1$ are obtained by plugging the result in \eqref{eq:definition:weights}. The order-zero weights $(k_i)$ are set to $1.$ In the approximation of the spectrum of $\mathscr{L}_1^{\up},$ we additionally perform a dimensionality reduction step by using the thresholded weights
\begin{align}
    \label{eq.thresholded.weights}
    k_{ij}^{(\alpha)} := k_{ij}\mathds{1}(k_{ij} \geq t^{\alpha-1}/(4\pi)), \ \text{ and } k_{ijk}^{\alpha} := k_{ijk}\mathds{1}(k_{ij}^{(\alpha)} k_{ik}^{(\alpha)} k_{jk}^{(\alpha)} > 0),
\end{align}
with $\alpha > 2$. The resulting weighted graph is visualized in Figure \ref{fig.data.cloud}. Since now $W_1$ and $W_2$ have zero entries, we restrict $\mathbb{R}^{d_1}$ and $\mathbb{R}^{d_2}$ and the matrices $W_1,W_2,B_1$ to their respective ranges. Finally, the positive spectrum of $\mathscr{L}_1^{\down}$ is obtained from that of the weighted unnormalized graph Laplacian of Example \ref{ex.graph.laplacian} with non-thresholded weights. \\

\noindent\textbf{Numerical procedure and results.} The experiments are implemented in Python. We use the function \verb|eigsh| from \verb|scipy| and obtain numerical approximations to the eigenvalues above. The up-Helmholtzian has a large kernel of dimension $\dim\ker(B_1).$ We remove the numerical eigenvalues that correspond to this kernel by discarding the numerical approximations smaller than $10^{-1}.$ The remaining numerical eigenvalues, sorted in ascending order, are denoted by $\hat\lambda_j^{\up}(t), \ j\geq 1.$ The same procedure applied to the numerical eigenvalues of the graph Laplacian results in $\hat\lambda_j^{\down}(t).$ We perform grid searches on the respective time parameters $t_u$ and $t_d$ of $\hat\lambda_j^{\up}$ and $\hat\lambda_j^{\down}$ and find approximate minimizers of 
\begin{align*}
    \mathcal{E}^{\down}(t_{d}) = \frac{1}{J}\sum_{j=1}^J\frac{|\lambda_j^{\down} - \hat \lambda_j^{\down}(t_{d})|}{\lambda_j^{\down}},\qquad \mathcal{E}^{\up}(t_{u}) = \frac{1}{J}\sum_{j=1}^J\frac{|\lambda^{\up}_j - \hat \lambda_j^{\up}(t_{u})|}{\lambda^{\up}_j},
\end{align*}
with $J = 8.$ Denote by $\hat t_d$ and $\hat t_u$ the resulting approximate minimizers. We set $n=700,$ and $\alpha=2.1.$ We set \verb|scipy|'s \verb|eigsh| to compute \verb|k = 450| numerical eigenvalues, with \verb|which = 'SA'| and \verb|tol = 1e-1|. $t^{\up}$ varies along $20$ regularly spaced values in $[0.4n^{-2/3}, 1.2n^{-2/3}]$ and $t^{\down}$ along $20$ regularly spaced values in $[0.1n^{-2/3}, 0.4n^{-2/3}].$ The experiment is repeated with $10$ different realizations of $X_1, \dots, X_n.$

The results are comparable to those of \cite{MR4452681}, where numerical simulations were run for the random-walk Laplacian on data from $ r\cdot S^2$. For $n=700,$ they obtain a minimal value of $\mathcal{E}^{\down}$ close to $10^{-0.92} \approx  0.120.$ We find $0.121\pm 0.010,$ showing close agreement between the two sets of results.

\begin{figure}[H]
\centering
\begin{tabular}{@{}ccc@{}}
    \includegraphics[width=0.31\textwidth, trim=3.5cm 9.5cm 3.5cm 9.5cm, clip]{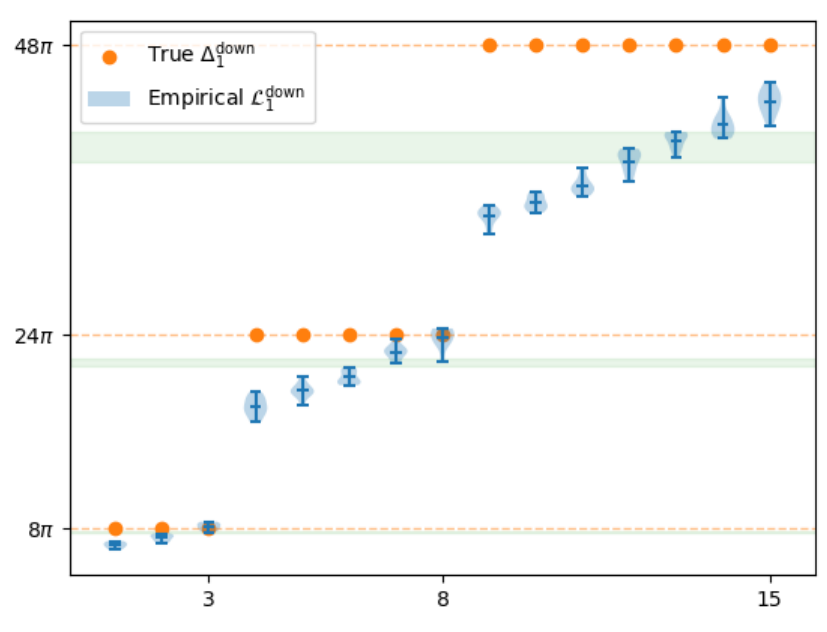} &
    \includegraphics[width=0.31\textwidth, trim=3.5cm 9.5cm 3.5cm 9.5cm, clip]{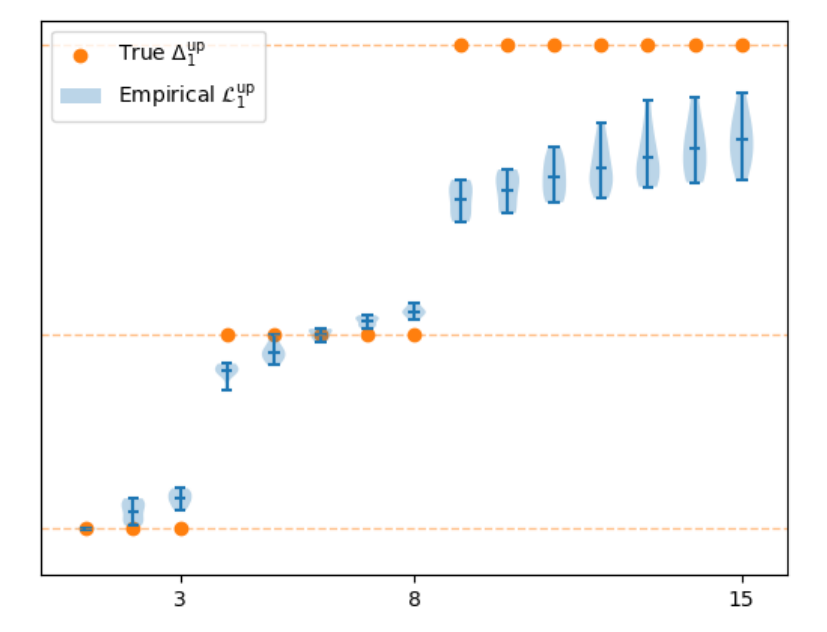} &
    \includegraphics[width=0.31\textwidth, trim=3.5cm 9.5cm 3.5cm 9.5cm, clip]{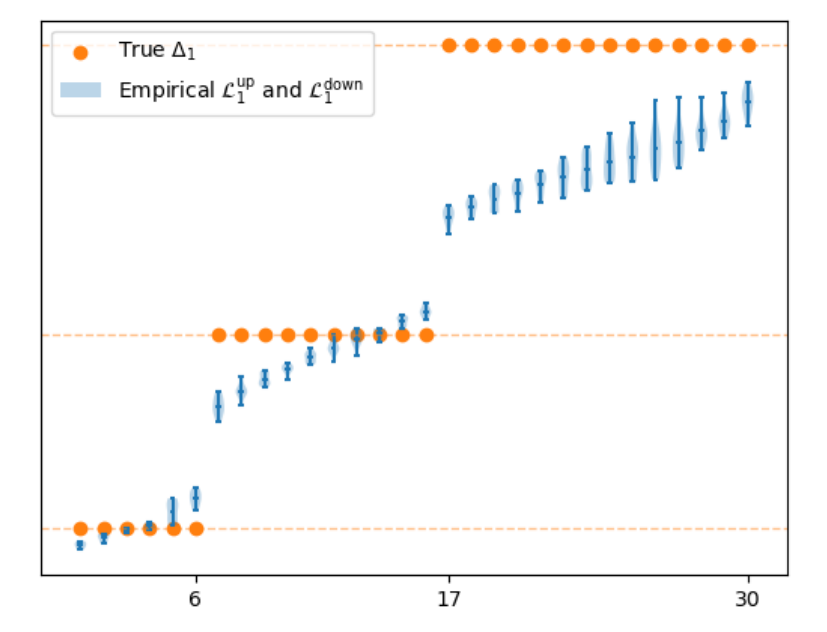}
\end{tabular}
\caption{Empirical and true eigenvalues of $r\cdot S^2.$ Index $i$ is associated with the violin plot of $\hat\lambda_i.$ The top, middle, and bottom horizontal bars are, respectively, the maximum, mean, and minimum over the $10$ realizations. \textbf{Left:} $15$ smallest eigenvalues $\hat\lambda_j^{\down}.$ \textbf{Center:} $15$ smallest eigenvalues $\hat\lambda_j^{\up}.$ \textbf{Right:} Bottom of the spectrum of $\mathscr{L}_1,$ obtained as the multiset union of the spectra shown in the left and center panels.}
\label{fig.spectra.sphere}
\end{figure}

\indent Figure \ref{fig.spectra.sphere} depicts the empirical distributions of each of the resulting eigenvalues. The empirical eigenvalues are clearly separated into clusters in all of the three plots. The gap between the first two clusters of empirical eigenvalues exceeds half of the corresponding population gap. We minimize $\mathcal{E}$ on the first two clusters of eigenvalues. The green bands in Figure \ref{fig.spectra.sphere} show the ranges of the first eigenvalue levels of the approximation $(1 - e^{-\hat t_d\lambda_k})/\hat t_d.$

\section*{Acknowledgements}
The research of JPL has been supported by a short-time postdoc fellowship at the Research Centre for Mathematical Modelling (RCM$^2$) in Bielefeld.

\bibliographystyle{plain}
\bibliography{references}

\end{document}